\documentclass[12pt,a4paper]{article}
\usepackage{amsfonts,amsmath,amssymb,enumerate,amsthm,ifthen,graphpap,longtable,time,xspace,graphicx}
\newcommand{\FigureSwitch}{on}
\newcommand{\CommentSwitch}{on}
\newcommand{\eop}{\hspace*{\fill}~$\square$} 
\theoremstyle{plain}
\numberwithin{equation}{section}
\newtheorem{theorem}{Theorem}[section]
\newtheorem{proposition}[theorem]{Proposition}
\newtheorem{lemma}[theorem]{Lemma}

\theoremstyle{definition}
\newtheorem{Remark}[theorem]{Remark}

\newcommand{\RootPath}{.}

\newcommand{\ExternalFiguresPath}{\RootPath/figures}

\newcommand{\mycaption}[1]{\centering{\vspace{\medskipamount}\refstepcounter{figure}\textbf{Figure~\thefigure.} {#1}}}

\newcommand{\natur}{\ensuremath{\mathbb{N}}}
\newcommand{\real}{\ensuremath{\mathbb{R}}}

\newcommand{\sprod}[2]{\left<#1,#2\right>}
\newcommand{\bigsprod}[2]{\bigl<#1,#2\bigr>}

\newcommand{\setcond}[2]{\left\{ #1 : #2 \right\}} % a set defined by a certain condition for their members
\newcommand{\setcondbegin}[2]{\left\{ #1 : #2 \right.} % a set defined by a certain condition for their members
\newcommand{\setcondend}[1]{\left.  #1 \right\}} % a set defined by a certain condition for their members
\newcommand{\bigsetcond}[2]{\bigl\{ #1 \,  :\, #2 \bigr\}}

\newenvironment{FigTab}[2]{
	\begin{figure}[htb]
	\setlength{\unitlength}{#2}
	\begin{center}
	\begin{tabular}{#1}
}{
    \end{tabular}
    \end{center}
    \end{figure}
}

\newcommand{\IncludeGraph}[2]{
\ifthenelse{\equal{\FigureSwitch}{on}}{
	\includegraphics[#1]{\ExternalFiguresPath/{#2}}
	}{
		\fbox{\texttt{#2}}
	}
}

\newcommand{\conv}{\mathop{\mathrm{conv}}\nolimits}

\newcommand{\aff}{\mathop{\mathrm{aff}}\nolimits}

\newcommand{\cone}{\mathop{\mathrm{cone}}\nolimits}
\newcommand{\rec}{\mathop{\mathrm{rec}}\nolimits}

\newcommand{\homog}{\mathop{\mathrm{hom}}\nolimits}

\newcommand{\diam}{\mathop{\mathrm{diam}}\nolimits}

\newcommand{\intr}{\mathop{\mathrm{int}}\nolimits}

\newcommand{\relint}{\mathop{\mathrm{relint}}\nolimits}

\newcommand{\E}{\mathop{\mathbb{R}}\nolimits} % Euclidean space

\newcommand{\ortproj}{\mathbin{|}}

\newcommand{\dotvar}{\,\cdot\,} % dot instead of variable

\newcommand{\MxL}{\left[} % Left matrix bracket
\newcommand{\MxR}{\right]} % Right matrix bracket
\newcommand{\ThmTitle}[2][]{\ifthenelse{\equal{#1}{}}{\emph{(#2)}}{\emph{(#2; #1)}}}
\newcommand{\notion}[2][]{\emph{#2}\xspace} % for the case one needs to create a subject index. 

\newcommand{\calF}{\mathcal{F}}
\newcommand{\calH}{\mathcal{H}}
\newcommand{\calA}{\mathcal{A}}
\newcommand{\calI}{\mathcal{I}}

\newcommand{\calC}{\mathcal{C}}

\newcommand{\uball}{\mathbb{B}}

\newcommand{\eps}{\varepsilon}

\newcommand{\vx}{\mathop{\mathrm{vert}}}

\newcommand{\mycite}[2]{\ifthenelse{\equal{#2}{}}{\cite{#1}}{\cite[#2]{#1}}\xspace}

\newcommand{\BosseGroetschelHenk}[1][]{\mycite{MR2166533}{#1}}
\newcommand{\Bernig}[1][]{\mycite{Bernig98}{#1}}
\newcommand{\Ziegler}[1][]{\mycite{MR1311028}{#1}}

\newcommand{\RAGbook}[1][]{\mycite{MR1659509}{#1}}

\newcommand{\BasuPollackRoy}[1][]{\mycite{MR2248869}{#1}}

\newcommand{\AveHenkSimPol}[1][]{\mycite{AveHenkRepSimplePolytopes}{#1}}

\newcommand{\SchnBk}[1][]{\mycite{MR94d:52007}{#1}} % Schneider's book

\newcommand{\comment}[1]{\ifthenelse{\equal{\CommentSwitch}{on}}{/* #1 */}{}}

\begin{document}
\title{Three-Dimensional Polyhedra can be Described \\ by Three  Polynomial Inequalities\footnote{Work supported by the German Research Foundation within the Research Unit 468 ``Methods from Discrete Mathematics for the Synthesis and Control of Chemical Processes''.}}
\date{}
\author{\small Gennadiy Averkov and Martin Henk}
\maketitle

\begin{abstract} 
	Bosse et al. conjectured that for every natural number $d \ge 2$ and every $d$-dimensional polytope $P$ in $\real^d$ there exist $d$ polynomials $p_0(x),\ldots,p_{d-1}(x)$ satisfying $P=\setcondbegin{x \in \real^d}{}$ $\setcondend{p_0(x) \ge 0, \ldots, p_{d-1}(x) \ge 0}.$ We show that for dimensions $d \le 3$ even every $d$-dimensional polyhedron can be described by $d$ polynomial inequalities. The proof of our result is constructive.
\end{abstract} 

\newtheoremstyle{itsemicolon}{}{}{\mdseries\rmfamily}{}{\itshape}{:}{ }{}
\newtheoremstyle{itdot}{}{}{\mdseries\rmfamily}{}{\itshape}{.}{ }{}
\theoremstyle{itdot}
\newtheorem*{msc*}{2000 Mathematics Subject Classification} % Theorem for Mathematics Subject Classification

\begin{msc*}
  Primary: 14P05, 52B11, 14Q99;  Secondary: 52A20
\end{msc*}
% 14P05 (semi-algebraic sets)
% 52B11 ($n$-dimensional polytopes), 52A20 (convex sets in $n$ dimensions).

\newtheorem*{keywords*}{Key words and phrases}

\begin{keywords*}
{\L}ojasiewicz's Inequality; polynomial; polytope; semi-algebraic set; Theorem of Br\"{o}cker and Scheiderer.
\end{keywords*}

\section{Introduction} 

 A subset $S$ of $\real^d$ ($d \in \natur$) is said to be an \emph{elementary closed semi-algebraic set} in $\real^d$ if 
\begin{equation} \label{S:def}
	S= (q_1,\ldots,q_m)_{\ge 0}:= \setcond{x \in \real^d}{q_1(x) \ge 0, \ldots, q_m(x) \ge 0}
\end{equation}
for some $m \in \natur$ and polynomials $q_1(x),\ldots,q_m(x)$  over $\real$.  Every elementary closed semi-algebraic set $S$ in $\real^d$ can be represented by $d \, (d+1)/2$ polynomial inequalities, i.e., there exist polynomials $p_1(x),\ldots,p_n(x)$ with $n \le d \,(d+1)/2$ such that 
\begin{equation}
	S=(p_1,\ldots,p_n)_{\ge 0}.
\end{equation} 
This follows from the well-known result of Br\"ocker and Scheiderer; see \RAGbook[Sections~6.5 and 10.4], \cite[p.~143]{MR1393194}. However, the arguments of Br\"ocker and Scheiderer are highly non-constructive, so that not much is known about the relationship of the polynomials $q_1(x),\ldots,q_m(x)$ and $p_1(x),\ldots,p_n(x)$, not to mention possible algorithms of determination of $p_1(x),\ldots,p_n(x)$ from $q_1(x),\ldots,q_m(x)$. One can consider the following more specific problem. Given a special class $\mathcal{S}$ of elementary closed semi-algebraic sets in $\real^d$, find the minimal $n$ such that every semi-algebraic set $S$ from $\mathcal{S}$ can be represented by $n$ polynomials $p_1(x),\ldots,p_n(x)$ and, moreover, find algorithms for determination of $p_1(x),\ldots,p_n(x)$ from $q_1(x),\ldots,q_m(x)$; see also \cite{Henk06PolRep}, \cite{AveRepElemSemiAlg}.

Results on representations of semi-algebraic sets are relevant for several research areas including polynomial and semi-definite optimization; see \cite{LaurentPolyOptimization}, \cite{HeltonNie08}. In this manuscript we are concerned with polynomial representations of polyhedra. A non-empty subset $P$ of $\real^d$ is said to be a \notion{polyhedron} if $P$ is the intersection of a finite number of closed halfspaces; see \cite{MR1311028}, \cite{MR2335496}. Furthermore, $P$ is called a \notion{polytope} if it is a convex hull of a finite point set in $\real^d$. It is well-known that polytopes can be characterized as bounded polyhedra.  The main result of the paper is the following theorem.

\begin{theorem} \label{main:thm} Let $d \in \{2,3\}$. Then there exists an algorithm that takes a $d$-dimensional polyhedron $P$  in $\real^d$ and constructs $d$ polynomials $p_0(x),\ldots,p_{d-1}(x) \in \real[x]$ satisfying $P=(p_0,\ldots,p_{d-1})_{\ge 0}$. \eop
\end{theorem} 

Bosse, Gr\"otschel, and Henk \BosseGroetschelHenk conjectured  that, for every $d \ge 2$, $d$ polynomial inequalities suffice for representing a $d$-dimensional polytope. Thus, Theorem~\ref{main:thm} confirms the above mentioned conjecture for $d \le 3$. We remark that the conjecture has recently been confirmed for simple polytopes of any dimension; see \cite{AveHenkRepSimplePolytopes}, \cite{AveRepElemSemiAlg}. Recall that a $d$-dimensional polytope $P$ is said to be \notion[simple polytope]{simple} if every vertex of $P$ is contained in precisely $d$ facets.

The paper is organized as follows.  Section~\ref{sect:lyrics} contains a sketch of the construction of the polynomials $p_0(x), \ldots, p_{d-1}(x)$ from Theorem~\ref{main:thm} and indicates some proof ideas. Sections~\ref{prelim:real:alg:geom} and \ref{sect:not:conv} present preliminaries from real algebraic geometry and convex geometry, respectively. In Section~\ref{sect:decomp:neighb} we introduce and study a special family of neighborhoods of a given convex polytope, which is used in the main proof. Section~\ref{sect:approx} is concerned with approximation of polytopes and polyhedral cones by algebraic surfaces. In Section~\ref{sect:proof} we prove Theorem~\ref{main:thm}.

\section{A sketch of the construction} \label{sect:lyrics}

We give a sketch of the construction of the polynomials $p_0(x),\ldots,p_{d-1}(x)$ from Theorem~\ref{main:thm}. Standard notions from convexity that are used in this section are introduced in Section~\ref{prelim:real:alg:geom} or can be found in the monographs \SchnBk, \cite{MR2335496}, \Ziegler.

Let $d \le 3$ and let $P$ be a $d$-dimensional polyhedron in $\real^d.$

\subsection{The case of polygons}

Consider the case when $P$ is a convex polygon in $\real^2$ (i.e., $d=2$ and $P$ is bounded). Let $m$ be the number of edges of $P.$ Then $P=\setcond{x \in \real^d}{q_1(x) \ge 0,\ldots, q_m(x) \ge 0}$ for appropriate affine functions $q_1(x),\ldots,q_m(x).$ We define $p_2(x) := q_1(x) \cdot \ldots \cdot q_m(x).$ One can construct a strictly concave polynomial $p_0(x)$ vanishing on each vertex of $P$. A possible construction of $p_0(x)$ was first given by Bernig; see \Bernig. The set $(p_0)_{\ge 0}$ is a strictly convex body whose boundary contains all vertices of $P$. It is intuitively clear that for $p_0(x)$ and $p_1(x)$ as above the equality $P=(p_0,p_1)_{\ge 0}$ is fulfilled; see also Fig.~\ref{simp:polyt:2d:fig}.

	\begin{FigTab}{c}{0.7mm}
	\begin{picture}(200,50)
	\put(60,2){\IncludeGraph{width=50\unitlength}{BernigSetP1.eps}}
	\put(0,2){\IncludeGraph{width=50\unitlength}{BernigSetP0.eps}}
	\put(150,2){\IncludeGraph{width=50\unitlength}{BernigP.eps}}
	\put(125,25){\IncludeGraph{width=12\unitlength}{RightArrow.eps}}
	\put(20,25){\scriptsize $(p_0)_{\ge 0}$}
	\put(80,25){\scriptsize $(p_1)_{\ge 0}$}
	\put(160,49){\scriptsize $P=(p_0,p_1)_{\ge 0}$}
%	\graphpaper[5](0,0)(200,50) %
	\end{picture}
	\\
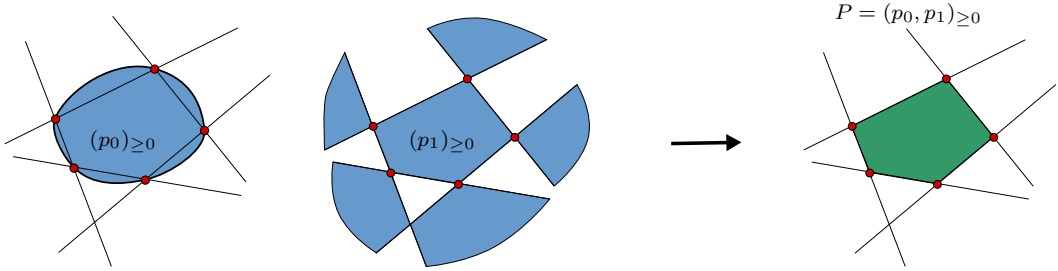
	\parbox[t]{0.98\textwidth}{\mycaption{Representing a polygon by two polynomial inequalities\label{simp:polyt:2d:fig}}} 
	\end{FigTab}

\subsection{The case of $3$-polytopes}

Consider the case when $P$ is a 3-polytope in $\real^3$ (i.e., $d=3$ and $P$ is bounded). Let $m$ be the number of facets of $P$. Then $P=\setcond{x \in \real^d}{q_1(x) \ge 0,\ldots, q_m(x) \ge 0}$ for appropriate affine functions $q_1(x),\ldots,q_m(x).$ We define $p_2(x):= q_1(x) \cdot \ldots \cdot q_m(x).$ Further on, we construct a strictly concave polynomial $p_0(x)$ vanishing on each vertex of $P$ and such that the distance between $P$ and $(p_0)_{\ge 0}$ is small enough; see also Fig.~\ref{octahedr:fig}\footnote{The figures in this subsection illustrate the construction for the case of the regular octahedron $\setcond{(\xi_1,\xi_2,\xi_3) \in \real^3}{|\xi_1|+|\xi_2|+|\xi_3| \le 1}.$ Since octahedra are not simple polytopes, they are not covered by the main result of \AveHenkSimPol on representation of simple polytopes by $d$ polynomial inequalities.}. With each vertex $v$ of $P$ we associate   a polynomial $b_v(x)$ such that $(b_v)_{\ge 0} = B(v) \cup \bigl(2 \, v-B(v)\bigr),$ where $B(v)$ is a pointed convex cone with apex at $v$ such that every edge of $P$ incident with $v$ lies in the boundary of $B(v)$; see Fig.~\ref{cone:at:vx:fig}. We define $p_1(x)$ as a ``weight\-ed combination'' of the polynomials $b_{v}(x)$ with $v$ ranging over the set of all vertices of $P$. More precisely, we set 
$$
	p_1(x) := \sum_{v} f_v(x)^{2k} \, b_{v}(x)
$$
where $v$ ranges over all vertices of $P$, $k \in \natur$ is suffiently large, and, for every vertex $v$ of $P$, $f_v(x)$ is a polynomial given by 
$$
	f_v(x) := \prod_{\substack{i=1,\ldots,m \\ q_i(v) \ne 0}} q_i(x),
$$
see also Fig.~\ref{surf:fig}.

\begin{FigTab}{cc}{0.7mm}
		\begin{picture}(80,60)
		\put(-10,-5){\IncludeGraph{width=90\unitlength}{Octahedron.eps}} 
		\put(20,45){$P$}
		\end{picture} 
		&
		\begin{picture}(80,60)
		\put(-10,-5){\IncludeGraph{width=90\unitlength}{OctahedronApprox.eps}} 
		\put(0,45){$p_0(x)=0$}
%		\put(50,35){$(g)_{\ge 0}$}
%		\graphpaper[2](0,0)(80,40)
		\end{picture} 
		\\
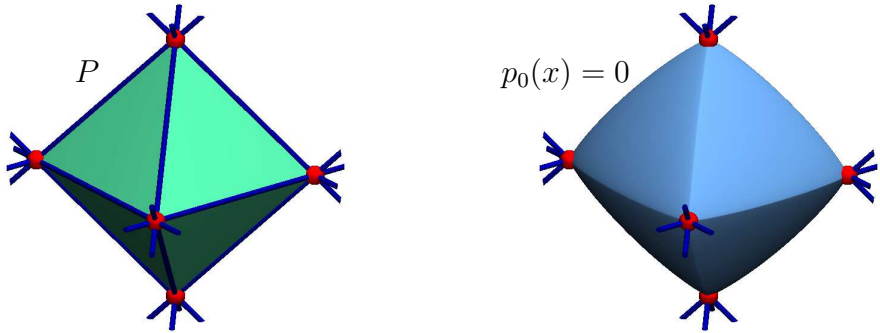
	\multicolumn{2}{c}{\parbox[t]{0.85\textwidth}{\mycaption{Regular octahedron and an approximation of the octahedron by a strictly concave polynomial surface\label{octahedr:fig}}}}
\end{FigTab}
%	\parbox[t]{0.45\textwidth}{\mycaption{Approximation of the octahedron by a sublevel set of a strictly concave polynomial\label{octahedr:approx:fig}}} 	

By construction $P \subseteq (p_0, p_1, p_2)_{\ge 0}.$ We have $p_0(x)<0$ for all $x$ sufficiently far away from $P.$ Further on, the region of all $x$ satisfying $p_2(x)<0$ is contiguous with all facets of $P.$ Finally, if $G$ is a vertex or an edge of $P$, $x_0$ is a point in the relative interior of $G$, and $v$ and $w$ are two vertices of $P$ with $v \in G$ and $w \notin G,$ respectively, then, for $x \rightarrow x_0$,  the term $f_v(x)^{2k} b_v(x)$ dominates over $f_w(x)^{2k} b_w(x)$ provided $k$ is sufficiently large. The latter observation is used for analysis of the properties of $(p_1)_{ \ge 0}.$

	\begin{FigTab}{cc}{0.7mm}
		\begin{picture}(80,60)
			\put(-10,-5){\IncludeGraph{width=90\unitlength}{OctahedronAndCone.eps}} 
			\put(65,30){$v$}
%		\graphpaper[2](0,0)(80,40)
		\end{picture} 
		&
		\begin{picture}(80,60)
			\put(-10,-5){\IncludeGraph{width=90\unitlength}{OctahedronConePlanes.eps}}
			\put(65,30){$v$}
		\end{picture} 
		\\
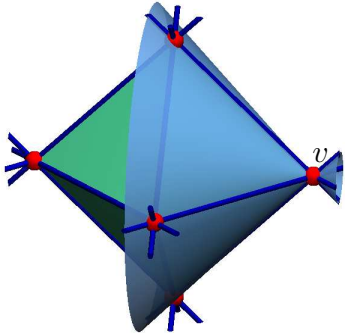
	\parbox[t]{0.45\textwidth}{\mycaption{Polyhedron $P$ and the surface given by the equation $b_{v}(x)=0$\label{cone:at:vx:fig}}} 
		&
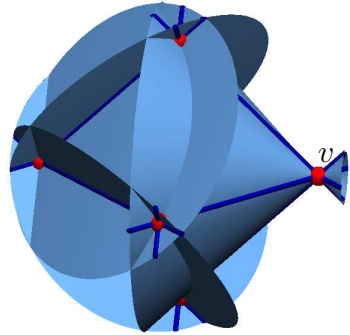
	\parbox[t]{0.45\textwidth}{\mycaption{The surface given by the equation $f_v(x)^{2k} b_{v}(x) =0$\label{surf:fig}}}
	\end{FigTab}

The geometry behind the construction is illustrated by the diagram from Fig.~\ref{diagr:octahedron}, where $P_J:= \setcond{x \in \real^3}{p_j(x) \ge 0 \ \forall j \in J}$ for $J \subseteq \{1,2,3\}$ with $J \ne \emptyset$ and an arrow between two sets indicates the inclusion relation.
	
\subsection{The case of unbounded polyhedra}

The case of unbounded $d$-polyhedra can be reduced to the case of bounded ones (at least for $d \le 3$). In Subsection~\ref{sect:proof:unbounded}	 we describe two possible ways of such a reduction. We sketch one of these ways below.

Assume that $P$ is unbounded. We replace $P$ by an isometric copy of $P$ in $\real^{d+1}$ with $o \not \in \aff P.$ It suffices to restrict considerations to line-free polyhedra $P$. The union of the conical hull of $P$ with the recession cone of $P$ is denoted by $\homog(P).$ The set $\homog(P)$ is a pointed polyhedral cone. We consider a hyperplane $H'$ in $\real^{d+1}$ such that $P' := H' \cap \homog(P)$ is bounded. Using the representations which were obtained above for the case of bounded polyhedra we find $d$ polynomials $p_0(x),\ldots,p_{d-1}(x)$ such that $$P' := \setcond{x \in \aff P'}{p_0(x) \ge 0,\ldots,p_{d-1}(x) \ge 0}.$$ It turns out that the choice of $p_0(x),\ldots,p_{d-1}(x)$ can be ``adjusted'' so that $p_0(x),\ldots,p_{d-1}(x)$ become homogeneous polynomials of even degree, while the set $(p_0)_{\ge 0}$ becomes the union of two pointed convex cones which are symmetric to each other with respect to the origin and which intersect precisely at $o.$ It follows that 
$$P = \setcond{x \in \aff P}{p_0(x) \ge 0,\ldots, p_{d-1}(x) \ge 0}.$$ 

\begin{FigTab}{c}{0.65mm}
\begin{picture}(132,155)
\put(33,53){\IncludeGraph{width=55\unitlength}{Octahedron.eps}} % \put(25,45)
\put(-17,28){\IncludeGraph{width=55\unitlength}{OctahedronP0.eps}}
\put(88,28){\IncludeGraph{width=55\unitlength}{OctahedronP1.eps}}
\put(33,118){\IncludeGraph{width=55\unitlength}{OctahedronP2.eps}}
\put(-17,88){\IncludeGraph{width=55\unitlength}{OctahedronP02.eps}}
\put(88,88){\IncludeGraph{width=55\unitlength}{OctahedronP12.eps}}
\put(33,-2){\IncludeGraph{width=55\unitlength}{OctahedronP01.eps}}
\put(5,10){\IncludeGraph{width=110\unitlength}{PDiagram.eps}}
\put(50,60){\scriptsize $P$}
\put(0,35){\scriptsize $P_0$}
\put(105,35){\scriptsize $P_1$}
\put(45,120){\scriptsize $P_2$}
\put(-3,95){\scriptsize $P_{0,2}$}
\put(105,95){\scriptsize $P_{1,2}$}
\put(50,5){\scriptsize $P_{0,1}$}
\end{picture} \\
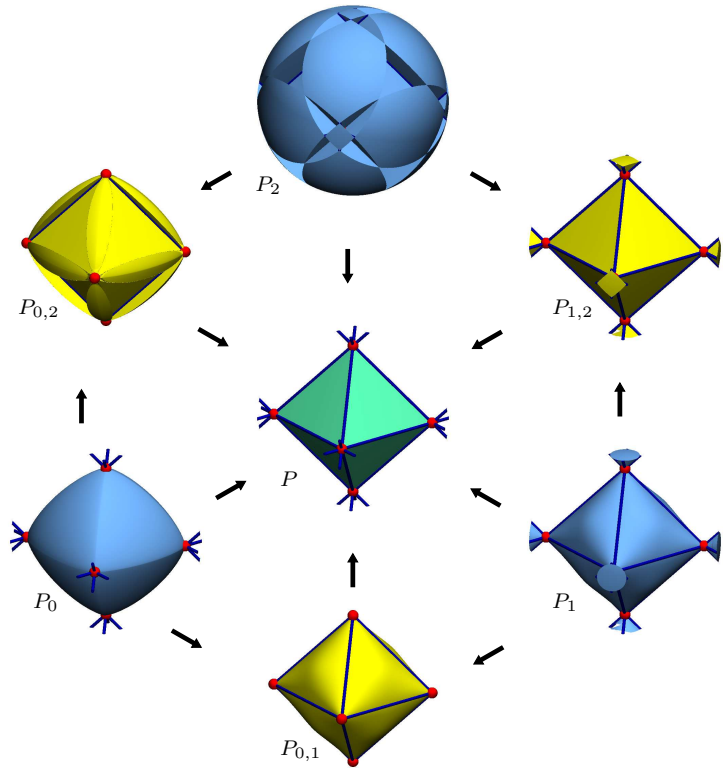
	\parbox[t]{0.95\textwidth}{\mycaption{Diagram illustrating representation of a three-dimensional polytope by three polynomial inequalities.\label{diagr:octahedron}}} 
\end{FigTab}

\section{Preliminaries from real algebraic geometry} \label{prelim:real:alg:geom} 

For information on real algebraic geometry  and, in particular, the geometry of semi-algebraic sets we refer to \cite{MR1393194}, \RAGbook. If $x$ is a variable in $\real^k$ ($k \in \natur$), then $\real[x]$ stands for the class of $k$-variate polynomials over $\real.$ Since $\real$ is a field of characteristic zero, we do not distinguish between real polynomials and real-valued polynomial functions.  In particular, we can view affine functions as polynomials of degree at most one. A subset $A$ of $\real^d$ is said to be \notion[semi-algebraic set]{semi-algebraic} if 
$$
	A=\bigcup_{i=1}^n \setcond{x \in \real^d}{f_{i,1}(x) > 0, \ldots, f_{i,s_i}(x)>0, \  g_i(x)=0},
$$
for some $n, s_1,\ldots, s_n \in \natur$ and $f_{i,j}(x), \, g_i(x) \in \real[x]$ with $i \in \{1,\ldots,n\}$ and $j \in \{1,\ldots,s_i\}.$

A real valued function $f(x)$ defined on a semi-algebraic set $A \subseteq \real^d$ is said to be a \notion{semi-algebraic function} if the graph of $f(x)$ is a semi-algebraic subset of $\real^{d+1}.$   The following theorem presents \notion{{\L}ojasiewicz's Inequality}; see \cite{MR0107168}, \RAGbook[Corollary~2.6.7].

\begin{theorem}  \ThmTitle{{\L}ojasiewicz 1959} \label{Loj}
	Let $A$ be a bounded and closed semi-algebraic set in $\real^d.$ Let $f(x)$ and $g(x)$ be continuous, semi-algebraic functions on $A$ satisfying $$\setcond{x \in A}{f(x) =0} \subseteq \setcond{x \in A}{g(x) = 0}.$$
	Then there exist $n \in \natur$ and $\lambda > 0$ such that 
	\begin{equation*}
		|g(x)|^n \le \lambda \, |f(x)| 
	\end{equation*}
	for every $x \in A.$  \eop
\end{theorem}

\newcommand{\cP}{\mathcal{P}}
An expression $\Phi$ is called a \notion{first-order formula  over $\real$} if $\Phi$ is a formula built with a finite number of conjunctions, disjunctions, negations, and universal or existential quantifiers on variables, starting from formulas of the form $f(x)=0$ or $f(x)>0$ with $f \in \real[x]$; see \RAGbook[Definition~2.2.3].  The \notion{free variables} of $\Phi$ are those variables, which are not quantified. A formula with no free variables is called a \emph{sentence}. Each sentence is either true or false. 
The following result is relevant for the constructive part of our main theorem; see \BasuPollackRoy[Algorithm~12.30]. 

\begin{theorem} \label{TarskiSeidenberg} \ThmTitle{Tarski 1951, Seidenberg 1954}
Let $\Phi$ be a sentence over $\real$. Then there exists an algorithm that takes $\Phi$ and decides whether $\Phi$ is true or false. \eop
\end{theorem}

 Dealing with algorithmic statements like Theorem~\ref{TarskiSeidenberg} (or Theorem~\ref{main:thm}), it assumed that a polynomial is given by its coefficients, a finite list of real coefficients occupies finite memory space, and arithmetic and comparison operations over real numbers are computable in one step. 

\section{Preliminaries from convexity} \label{sect:not:conv}

\newcommand{\dist}{\mathop{\mathrm{dist}}\nolimits}

For information on convex bodies and polytopes we refer to \SchnBk, \cite{MR2335496}, \Ziegler. The cardinality of a set is denoted by $| \dotvar |.$ The origin, Euclidean norm, and scalar product in $\real^d$ are denoted by $o$,  $\|\dotvar\|,$ and $\sprod{\dotvar}{\dotvar},$ respectively. We endow $\real^d$ with its Euclidean topology. By $\uball^d(c,\rho)$ we denote the closed Euclidean ball in $\real^d$ with center at $c \in \real^d$ and radius $\rho >0.$ The notations $\aff$, $\conv$, $\intr$, and $\relint$ stand for the affine hull, convex hull, interior, and relative interior, respectively.  If $X$ and $Y$ are non-empty sets in $\real^d$ we set $X \pm Y:= \setcond{x \pm y}{x \in X, \ y \in Y}$. The \emph{Hausdorff distance} $\dist(X,Y)$ between non-empty, compact sets $X, \, Y \subseteq \real^d$ is defined by 
$$
	\dist(X,Y):= \max \Bigl\{ \max_{x \in X} \min_{y \in Y} \|x-y\| \, , \, \max_{y \in Y} \min_{x \in X} \|x-y\| \Bigr\},
$$
see also \SchnBk[p.~48]. It is known that $\dist(X,Y)$ is a metric on the space of non-empty, compacts sets in $\real^d.$  The Hausdorff distance can be expressed by
\begin{equation} 
	\dist(X,Y) := \min \setcond{ \rho \ge 0}{X \subseteq Y+\uball^d(o,\rho), \ Y \subseteq X+\uball^d(o,\rho)}.
\end{equation} 

\newcommand{\uu}[2]{u_{#1}(#2)}
\newcommand{\fc}[2]{\mathcal{F}_{#1}(#2)}
\newcommand{\fcst}[3]{\mathcal{F}_{#1}(#2,#3)} 

If for a convex set $X \subseteq \real^d$ and a point $x_0 \in \real^d$ one has $x_0 + \alpha \, x \in X$ for every $\alpha \ge 0$ and $x \in X,$ then $X$ is said to be a \notion{convex cone} with \notion{apex} at $x_0.$ Given a non-empty subset $X$ of $\real^d$, we introduce the \notion{conical hull} $\cone (X)$ of $X$ as the set of all possible linear combinations $\lambda_1 \, x_1 + \cdots + \lambda_m \,  x_m$ with $m \in \natur,$ $\lambda_i \ge 0$, and $x_i \in X$ for $i \in \{1,\ldots,m\}.$ Clearly, $\cone (X)$ is a convex cone with apex at the origin. 

The notions of polyhedron and polytope were defined in the introduction. Polyhedra (resp. polytopes) of dimension $n$ are referred to as \notion[polyhedron]{$n$-polyhedra} (resp. \notion[polytope]{$n$-polytopes}). A subset of $\real^d$ which is both a cone and polyhedron is said to be a \notion{polyhedral cone.} A polyhedron is said to be \notion{line-free} if it does not contain a straight line. The set $\rec(P):= \setcond{u \in \real^d}{P +u \subseteq P}$ is said to be the \notion{recession cone} of $P$.  A line-free polyhedral cone is said to be \notion{pointed}. Every polyhedron $P$ can be represented by 
\begin{equation} \label{polyt:decomp}
	P = Q+ L + C,
\end{equation} 
where $Q$ is a polytope, $L$ is a linear subspace of $\real^d$, and $C$ is a pointed polyhedral cone; see \Ziegler[Section~1.5]. For \eqref{polyt:decomp} one necessarily has $L+C=\rec(P).$ If $o \notin \aff P,$ we introduce 
$$
	\homog(P) := \cone(P) \cup \rec(P).
$$
It is known that $\homog (P)$ is a polyhedral cone; see \Ziegler[pp.~44-45]. Furthermore, it can be verified that if $P$ is line-free, then $\homog(P)$ is pointed. 

For a convex polytope $P$ in $\real^d$ we introduce the \notion{support function} and \notion{exposed facet} in direction $u$ by the equalities 
$$h(P,u):= \max \setcond{\sprod{x}{u}}{x \in P}$$ 
and 
$$F(P,u):= \setcond{x \in P}{\sprod{x}{u}=h(P,u)},$$ 
respectively.  Given a polytope $P$ in $\E^d$, $\calF(P)$ denotes the class of all faces of $P$, and, for $i \in \{-1,\ldots,d\}$, $\calF_i(P)$ stands for the class of all \notion{$i$-faces} of $P$ (i.e., faces of dimension $i$). We recall that for every point $x$ of $P$ there exists a unique non-empty face $G$ of $P$ with $x \in \relint P.$ Consequently, $P$ is the disjoint union of the relative interiors of all non-empty faces of $P.$ A face $G$ of an $n$-polytope $P$ in $\real^d$ is said to be a \notion{facet} of $P$ if $G$ has dimension $n-1.$ If $G$ is a face of $P$, we define $$\calF_{i}(G,P):=\setcond{F \in \calF_i(P)}{G \subseteq F}.$$  By $\vx(P)$ we denote the set of all \notion{vertices} of $P$ (i.e., the set of $0$-dimensional faces). If $F$ is a facet of $P$ by $\uu{F}{P}$ we denote the unit normal of $P$ which is parallel to $\aff (P)$ and satisfies $F(P,\uu{F}{P}) = F.$ More generally, if $G$ is a proper face of $P$, we define
$$\uu{G}{P} := \biggl( \sum_{F \in \fcst{n-1}{G}{P}} \uu{F}{P} \biggr) \biggl/ \biggl\| \sum_{F \in \fcst{n-1}{G}{P}} \uu{F}{P} \biggr\| \biggr. ,$$ 
where $n:=\dim (P).$ It is not hard to see that $F(P,\uu{G}{P})=G.$  

With every facet $F$ of $P$ we associate the affine function
\begin{equation*}
	q_F(P,x)  := \frac{h(P,u_F(P))-\sprod{u_F(P)}{x}}{\diam(P)}, 
\end{equation*}
where  
\begin{equation*}
	\diam(P):= \max \setcond{ \|x-y\|}{x, \, y \in P} = \max \setcond{\|x-y\|}{x, \, y \in \vx(P)}.
\end{equation*}
By construction, $0 \le q_F(P,x) < 1$ for every $x \in P$ with equality $q_F(P,x)=0$ if and only if $x \in F.$

For $v \in \vx(P),$ the \notion{supporting cone} $S(P,v)$  of $P$ at $v$ is given by 
\begin{equation*} 
	S(P,v):= \cone (P-v),
\end{equation*}
see also \SchnBk[p.~70].  If $\dim P=d,$ then 
\begin{equation*}
S(P,v) = \setcond{x \in \real^d}{\bigsprod{u_F(P)}{x} \le 0 \ \forall \, F \in \calF_{d-1}(v,P)}
\end{equation*}
and by this 
\begin{equation*} 
	v+ S(P,v)  = \setcond{x \in \real^d}{q_F(P,x) \le 0 \ \forall \, F \in \calF_{d-1}(v,P)}.
\end{equation*}
For every $v \in \vx(P)$ we define the hyperplane 
$$H_v(P) := \setcond{x \in \real^d}{\bigsprod{x}{u_v(P)} =-1}$$ and the polytope $$P_v:= S(P,v) \cap H_v(P).$$ 

It can be shown that $\cone(P_v)=S(P,v).$ Hence the cone
\begin{align*}
	S_\rho(P,v) := \cone \Bigl(\bigl(P_v + \uball^d(o,\rho) \bigr) \cap H_v(P) \Bigr).
\end{align*}
with $\rho>0$ and $v \in \vx(P)$ can be viewed as an ``outer approximation'' of the supporting cone with the parameter $\rho$ controlling the quality of the approximation.
For every $v \in \vx(P)$ one has
\begin{equation} \label{08.04.29,09:54}
	P \setminus \{v\} \subseteq \intr \bigl( v+S_\rho(P,v) \bigr).
\end{equation}

  For $x \in \real^d$ and a polytope $P$ in $\real^d$ and $x \in \aff P$ we introduce the notation
	$$
		\calF^-(P,x):= \setcond{F \in \calF_{d-1}(P)}{q_F(P,x) \le 0}.
	$$
The class $\calF^-(P,x)$ be interpreted as the set of all facets of the polytope $P$ which are visible from $x$.

 In Theorem~\ref{main:thm} and the remaining statements dealing with algorithms a polyhedron is assumed to be given by a system of affine inequalities (the so-called \notion{H-representation}) and a polynomial by a list of its coefficients. It is not difficult to show that there exists an algorithm that takes a polytope $P$ and constructs all functions $q_F(P,x)$ where $F$ is a facet of $P.$

\section{A family of neighborhoods of a $3$-polytope} \label{sect:decomp:neighb}

Let $\eps>0,$ $\rho>0$, and $P$ be a $3$-polytope in $\real^3$. 
We set $U_P(P,\eps,\rho):=P$ and $U_F(P,\eps,\rho):= \bigsetcond{x \in \real^3}{\calF^-(P,x)=\{F\}}$ for $F \in \calF_2(P).$ Furthermore, for $I \in \calF_1(P)$ and  $v \in \vx(P)$ we define 
\begin{eqnarray} 
	U_I(P,\eps,\rho) &:=& \bigl(I+\uball^3(o,\eps) \bigr) \cap \setcond{x \in \real^3}{\calF_2(I,P) \subseteq \calF^-(P,x)} \nonumber \\ 
		& & \cap  \biggl( \bigcap_{v \in \vx(P)} S_\rho(P,v) \biggr), \label{UI:def} \\
	U_v(P,\eps,\rho) &:= & \uball^3(v,\eps) \setminus \bigl(v + \intr S_\rho(P,v)),  
\end{eqnarray}
see also Figs.~\ref{EdgeRegionsAndOctahedron:fig} and \ref{VertexRegionsAndOctahedron:fig}.
We also introduce the region 
$$
	U_v'(P,\eps,\rho) := \uball^3(v,\eps) \setminus \Bigl( (v + \intr S_\rho(P,v)) \cup (v- \intr S_\rho(P,v)) \Bigr), \label{Uv:def} 
$$
which is a subset of $U_v(P,\eps,\rho)$; see Fig.~\ref{ModVertexRegionsAndOctahedron:fig}.  With $P$ we associate the set $U(P,\eps,\rho)$ 
$$
	U(P,\eps,\rho) := \bigcup_{G \in \calF(P) \setminus \{ \emptyset \}} U_G(P,\eps,\rho),
$$

\begin{FigTab}{ccc}{0.6mm}
		\begin{picture}(70,60)
		\put(-5,0){\IncludeGraph{width=77\unitlength}{EdgeRegionsAndOctahedron.eps}} 
%		\graphpaper[5](0,0)(100,60)
		\end{picture}  
		&
		\begin{picture}(70,60)
		\put(-5,0){\IncludeGraph{width=77\unitlength}{VertexRegionsAndOctahedron.eps}} 
%		\graphpaper[5](0,0)(100,60)
		\end{picture}  
		&
		\begin{picture}(70,60)
		\put(-5,0){\IncludeGraph{width=77\unitlength}{ModVertexRegionsAndOctahedron.eps}} 
%		\graphpaper[5](0,0)(100,60)
		\end{picture}  
		\\
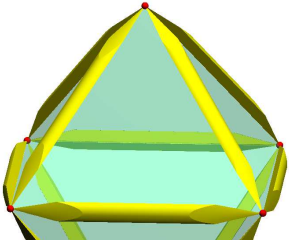
	\parbox[t]{0.3\textwidth}{\mycaption{The regions $U_I(P,\eps,\rho)$ with $I \in \calF_1(P)$\label{EdgeRegionsAndOctahedron:fig}}} & \parbox[t]{0.3\textwidth}{\mycaption{The regions $U_v(P,\eps,\rho)$ with $v \in \vx(P)$ \label{VertexRegionsAndOctahedron:fig}}} &\parbox[t]{0.3\textwidth}{\mycaption{The regions $U'_v(P,\eps,\rho)$ with $v \in \vx(P)$ \label{ModVertexRegionsAndOctahedron:fig}}}
\end{FigTab}

The main statement of this section is Propsition~\ref{U:prop}. The statement of this proposition is intuitively clear. However, we are not aware of any short proof of it.
\begin{proposition} \label{U:prop}
	Let $P$ be a $3$-polytope in $\real^3.$ Then $ P \subseteq \intr U(P,\eps,\rho)$
	for all sufficiently small $\rho>0$ and all $\eps> 0.$  \eop
\end{proposition} 

The proof of Proposition~\ref{U:prop} is based on a number of auxiliary statements. The proofs of the simple Lemmas~\ref{sem-cont-f-minus}, \ref{vis:vx:lem}, and \ref{08.04.30,11:28} are omitted.
	
\begin{lemma} \label{sem-cont-f-minus} Let $P$ be a $d$-polytope in $\real^d,$ $G$ a non-empty face of $P$, and $x \in \relint G.$ Then there exists $\eps>0$ such that every $y \in \uball^d(x,\eps)$ satisfies $\calF^-(P,y) \subseteq \calF_{d-1}(G,P)$  \eop
\end{lemma} 
% \begin{proof}
% 	 Assume the contrary. Then there exists a sequence $(x_n)_{n=1}^{\infty}$ of points from $\real^d$ that converges to $x$ and satisfies $\calF^-(P,x_n) \not\subseteq \calF_{d-1}(G,P)$ for every $n \in \natur.$ Since $\calF_{d-1}(P)$ is finite, there exists $F \in \calF_{d-1}(P)$ such that $\natur^\ast := \setcond{n \in \natur}{F \in \calF^-(P,x_n) \setminus \calF_{d-1}(G,P)}$ is infinite.
% 	 Taking the limit as $n \rightarrow \infty$ in the inequality  $q_F(P,x_n) \le 0$, which holds for every  $n \in \natur^\ast,$ we obtain $q_F(P,x) \le 0.$ Since $x \in P$ we even have $q_F(P,x)=0.$ Consequently, $F \in \calF_{d-1}(G,P),$ a contradiction.
% \end{proof}

If a point $x$ outside $P$ is sufficiently close to $P$, then all facets of $P$ visible from $x$ have a vertex in common. This is stated in a more formal way in the following lemma.

\begin{lemma} \label{vis:vx:lem} 
	Let $P$ be a $d$-polytope in $\real^d$. Then for some $\eps>0$ and every $x \in P+\uball^d(o,\eps)$ there exists  $v \in \vx (P)$ satisfying $\calF^-(P,x) \subseteq \calF_{d-1}(v,P).$ \eop
\end{lemma} 

\begin{lemma} \label{08.04.30,11:28} 
	Let $P$ be a $d$-polytope in $\real^d$ with $o \in \vx(P)$ and let  $F \in \calF_{d-1}(o,P)$. Then $G:= \aff F  \cap P_o$ is a facet of $P_o$. Furthermore, for some $\alpha>0$ one has
	\begin{equation*} 
		q_F(P,x) = \alpha \cdot q_G(P_o,x) \qquad \forall \, x \in H_o(P).
	\end{equation*}
	\eop
\end{lemma}

\begin{lemma} \label{rho:lem}
	Let $P$ be a $3$-polytope in $\real^3$. Then there exists $\rho>0$ such that for all $v \in \vx(P)$ and $x \in \bigl( v+S_{\rho}(P,v) \bigr) \setminus \{v\}$ one can find $I \in \calF_1(v,P)$ satisfying
	\begin{equation} 
		\calF^-(P,x) \cap \calF_2(v,P)  \subseteq \calF_2(I,P). \label{card:2:bound} 
	\end{equation} 
\end{lemma} 
\begin{proof} 
Consider an arbitrary $v \in \vx(P).$ Replacing $P$ by its appropriate translation we assume that $v=o.$ By Lemma~\ref{vis:vx:lem} applied to $P_o,$ one can choose $\rho>0$ such that for every $x \in H_o(P) \cap ( P_o+\uball^3(o,\rho) )$ there exists $w \in \vx(P_o)$ satisfying $\calF^-(P_o,x) \subseteq \calF_1(w,P_o).$ Take an arbitrary $x \in S_\rho(P,o) \setminus \{o\}$. Then $\sprod{x}{u_o(P)} <0,$ and we introduce $$y:= - \frac{x}{\sprod{x}{u_o(P)}} \in H_o(P) \cap \bigl(P_o+\uball^3(o,\rho)\bigr).$$ By the choice of $\rho$, there exists $w \in \vx(P_v)$ with $\calF^-(P_o,y) \subseteq \calF_1(w,P_o).$  In view of Lemma~\ref{08.04.30,11:28}, the latter inclusion implies $\calF^-(P,x) \cap \calF_2(o,P)  \subseteq \calF_2(I,P)$, where $I \in \calF_1(o,P)$ is such that $H_o(P) \cap \cone(I)=\{w\}.$ This shows \eqref{card:2:bound}. 
\end{proof} 

Now we are ready to prove the main statement of this section. 

\begin{proof}[Proof of Proposition~\ref{U:prop}] 
	We pick an arbitrary $x \in P$ and show that $x \in \intr U(P,\eps,\rho).$ 	By $G$ we denote the unique face of $P$ with $x \in \relint G.$ If $G=P,$ then $x \in \intr P$ and by this $x \in \intr U(P,\eps,\rho)$. If $G$ is a facet of $P$, then, in view of Lemma~\ref{sem-cont-f-minus}, all points sufficiently close to $x$ lie in $P \cup U_G(P,\eps,\rho).$ If $G$ is an edge of $P$, we show that
	\begin{equation} \label{08.06.18,09:19}
		 x \in \intr \Biggl( P \cup \biggl( \bigcup_{F \in \calF_2(G,P)} U_F(P,\eps,\rho) \biggr)  \cup U_G(P,\eps,\rho) \Biggr).
	\end{equation}
	In fact, it is easy to see that for every $v \in \vx(P)$ one has $x \in v+\intr S_\rho(P,v)$ and $x \in \intr (G+\uball^3(o,\eps)).$ Furthermore, applying Lemma~\ref{sem-cont-f-minus} we obtain 
	$$x \in \intr  \setcond{y \in \real^3}{\calF^-(F,y) \subseteq \calF_2(G,P) } .$$
	Thus, \eqref{08.06.18,09:19} is fulfilled, and by this $x \in \intr U(P,\eps,\rho).$ 

	It remains to consider the case $x=v$ for some $v \in \vx(P).$ We assume that $\rho$ is small enough so that the assertion of Lemma~\ref{rho:lem} is fulfilled. We show that $x \in \intr U(P,\eps,\rho)$ by contradiction. Assume that there exists a sequence $(x_n)_{n=1}^{\infty}$ converging to $x$ and such that $x_n \not\in U(P,\eps,\rho)$ for every $n \in \natur$.  Take an arbitrary $\calF^\ast \subseteq \calF_2(P)$ such that $\natur^\ast := \setcond{n \in \natur}{\calF^-(P,x_n)=\calF^\ast}$ is infinite. By Lemma~\ref{sem-cont-f-minus}, for all sufficiently large $n \in \natur$ one has $\calF^-(P,x_n) \subseteq \calF_2(v,P).$ Hence $\calF^\ast \subseteq \calF_2(v,P)$. If $|\calF^\ast|=1,$ then $x_n \in U(P,\eps,\rho)$ for every $n \in \natur^\ast.$ If $|\calF^\ast| \ge 3$, then for all sufficiently large $n$  one has $x_n \not \in v+S_\rho(P,v)$ and by this $x_n \in \uball^3(v,\eps) \setminus (v + \intr S_\rho(P,v))=U_v(\eps,\rho).$  In fact, if we assume the contrary, i.e., $x_n \in v+S_\rho(P,v)$ for all sufficiently large $n \in \natur^\ast,$ then, by \eqref{card:2:bound}, we see that  for some $I \in \calF_1(v,P)$ one has $\calF^\ast \cap \calF_2(v,P) \subseteq \calF_2(I,P).$ In view of $\calF^\ast \subseteq \calF_2(v,P),$ the latter implies $\calF^\ast \subseteq \calF_2(I,P).$ Hence $|\calF^\ast| \le 2$, a contradiction. We consider the remaining case $|\calF^\ast|=2.$ In view of \eqref{08.04.29,09:54}, one has $x_n \in S_\rho(P,w)$ for all sufficiently large $n$ and every $w \in \vx(P) \setminus \{v\}.$ It is easy to see that for all sufficiently large $n$ one has $x_n \in \uball^3(v,\eps)$. Hence, for all large $n$ one has $x_n \in U_v(P,\eps,\rho)$ if $x \not\in v+ \intr S_\rho(P,v)$ and, taking into account the assertion of Lemma~the assertion of Lemma~\ref{rho:lem} and the definition of $U_I(P,\eps,\rho)$, one has $x_n \in U_I(P,\eps,\rho)$ for some $I \in \calF_1(v,P)$ provided $x \in v + S_\rho(P,v)$. Consequently for all large $n$ one has $x_n \in U(P,\eps,\rho).$ This yields the assertion of the proposition.
\end{proof} 

\section{Approximation of polytopes and polyhedral cones} \label{sect:approx}

A polynomial $f \in \real[x]$ of degree $n$ is said to be \notion{homogeneous} if $f(\alpha \, x ) = \alpha^n \, f(x)$ for all $\alpha \in \real$ and $x \in \real^d$. Consider a hyperplane $H$ in $\real^d$ with $o \not\in H.$ Obviously, $H$ can be represented by 
\begin{equation} \label{plane:vec:eq}
	H = \setcond{x \in \real^d}{\sprod{x}{u} =1}
\end{equation}
 with $u \in \real^d \setminus \{o\}$ uniquely determined by $H.$ Consider a polynomial $f \in \real[x]$ of degree $n$. The \notion{homogeneous continuation} of  $f(x)|_H$ (i.e., the restriction of $f(x)$ to $H$) to $\real^d$ is defined as the polynomial $\Tilde{f} \in \real[x]$ uniquely determined by $\Tilde{f}(x)= \sprod{x}{u}^n f\bigl( x / \sprod{x}{u} \bigr)$ with  $x \in \real^d \setminus H.$  Clearly, $\Tilde{f}$ is a homogeneous polynomial.

We say that there is an algorithm that constructs a sequence $\bigl(p_l(x)\bigr)_{l=1}^{\infty}$ of polynomials from the given input if there exists an algorithm which constructs $p_l(x)$ from $l \in \natur$ and the given input. Theorem~\ref{tight:interp} and Proposition~\ref{b:prop} are the main statements of this section. The former is concerned with a special type of approximation of polytopes and the latter with approximation of supporting cones of polytopes. Given a polytope $P$ in $\real^d,$ we introduce the following conditions on a sequence $\bigl(g_l(x)\bigr)_{l=1}^{\infty}$ of polynomials from $\real[x]$.
\begin{itemize} 
	\item[$\calA(P)$:] The Hausdorff distance from $P$ to
	$\setcond{x \in\aff P}{g_l(x) \ge 0}$ converges to zero, as $l \rightarrow \infty$.
	\item[$\calI(P)$:] For every $v \in \vx(P)$ and $l \in \natur$ one has $g_l(v) = 0.$
	\item[$\calC(P)$:]  For every  $l \in \natur$ the function $g_l(x)|_{\aff P}$  is strictly concave.
	\item[$\calH(P)$:] For every $l \in \natur$ one has 
	\begin{equation*} 
		\setcond{x \in \real^d}{\Tilde{g}_l(x) \ge 0, \  \sprod{x}{u} \ge 0} = \cone \bigl(\setcond{x \in \aff P}{g_l(x) \ge 0}\bigr),
	\end{equation*}
	where $\Tilde{g}_l(x)$ is a homogeneous continuation of $g_l(x)|_{\aff P}.$ 
\end{itemize} 

The notations $\calA,$ $\calI$, $\calC$, and $\calH$ are derived from the words `approximation', `interpolation', `concavity', and `homogeneity', respectively. Condition~$\calH(P)$ makes sense only for the case when $o \not\in \aff P.$ 

\begin{theorem} \label{tight:interp} Let $P$ be  a polytope in $\real^d$. Then the following statements hold true. 
\begin{enumerate}[I.]
	\item \label{exists:AIC} There exists an algorithm that takes $P$ and constructs a sequence $(g_l(x))_{l=1}^{\infty}$ of polynomials from $\real[x]$ that satisfy $\calA(P),$ $\calI(P),$ and $\calC(P).$ 
	\item \label{AIC:impl:rho} For every sequence $\bigl(g_l(x)\bigr)_{l=1}^{\infty}$ satisfying $\calA(P),$ $\calI(P),$ and $\calC(P)$ there exists $\rho>0$ with 
	\begin{equation} \label{cones:on:vert:cond}
		\bigl(v - S_\rho(P,v) \bigr) \cap \setcond{x \in \aff P}{g_l(P, x) \ge 0}  = \{v\} \qquad \forall \, v \in \vx(P) \ \forall \, l \in \natur.
	\end{equation} 
	\item \label{exists:AICH} If $\dim (P)=d-1$ and $o \not\in \aff P,$ there exists an algorithm that takes $P$ and constructs a sequence $\bigl(g_l(x)\bigr)_{l=1}^{\infty}$ of polynomials from $\real[x]$ that satisfy $\calA(P),$ $\calI(P),$ $\calC(P)$, and $\calH(P).$ 
\end{enumerate}
\end{theorem}
\begin{proof} Let $n:=\dim (P).$

\emph{Part~\ref{exists:AIC}.} For $l \in \natur$ we define $g_l(x)$ by 
	\begin{equation} \label{gk:def}
		g_l(x) := 1 - \sum_{v \in \vx(P)} y_{v,l} \Biggl( \frac{1}{|\calF_{n-1}(v,P)|} \sum_{F \in \calF_{n-1}(v,P)} \bigl(1-q_F(P,x) \bigr)^{2(l+l_0)} \Biggr)^{2(l+l_0)}
	\end{equation}
where $l_0 \in \natur$ and $y_{v,l} \in \real$. In \cite{AveHenkRepSimplePolytopes} it was shown that an appropriate  $l_0 \in \natur$ and scalars $y_{v,l}$ with $v \in \vx(P)$ and $l \in \natur$ can be constructed such that the sequence $\bigl(g_l(x)\bigr)_{l=1}^{\infty}$ satisfies $\calI(P)$ an $\calA(P).$ By construction, $\bigl(g_l(x)\bigr)_{l=1}^{\infty}$ also satisfies the condition $\calC(P)$.

 \emph{Part~\ref{AIC:impl:rho}.} Without loss of generality we may assume that $n=d.$ Consider an arbitrary $v \in \vx(P).$ Replacing $P$ by an appropriate translation we assume that $v=o.$ Let $\rho'_o:= \min \setcond{\|x-y\|}{x \in P, y \in (-P_o)}.$ We show that there exists $l'_o \in \natur$ such that  one has
\begin{equation} \label{08.06.18,10:15}
	\forall \, l \ge l_o' \ \forall \, x \in \bigl( -P_o + \uball^d(o, \rho'_o / 2 ) \bigr) \cap (-H_o(P)) \ : \ g_l(x)<0.
\end{equation} 
 Assume the contrary. Then there exist sequences $(l_k)_{k=1}^{\infty},$ $(x_k)_{k=1}^{\infty}$, and $(y_k)_{k=1}^{\infty}$ such that $l_k \in \natur,$ $x_k \in (-H_o(P))$, $y_k \in -P_o$, $\|x_k-y_k \| \le \rho'_o / 2$, and $g_{l_k}(x_k) \ge 0$ for every $k \in \natur$. Then, for every $k \in \natur,$ we obtain
\begin{align*}
	  \dist \bigl(P, \setcond{x \in\aff P}{g_{l_k}(x) \ge 0} \bigr) & \ge \min_{x \in P} \|x-x_k \| \\  & \ge \min_{x \in P} \|x-y_k\|-\|x_k - y_k\| \\ & \ge 	 \rho'_o - \|x_k-y_k\|  \ge \rho'_o/2 > 0,
\end{align*}
a contradiction to the assumption that  $\bigl(g_l(x)\bigr)_{l=1}^{\infty}$ satisfies $\calA(P)$. Let us fix $l'_o$ satisfying \eqref{08.06.18,10:15}. Next we show that for every for every $l \in \natur$ and every $x \in -P_o$ one has $g_l(x)  < 0.$ Assume the contrary, i.e. $g_l(x) \ge 0.$ Since 
$\bigl(g_l(x)\bigr)_{l=1}^{\infty}$ satisfies $\calC(P)$ and $\calI(P)$ we deduce that $g_l(y) > 0$ for every $y \in P \setminus \vx(P).$ For a sufficiently small $\eps>0$ one has $-\eps \, x \in P \setminus \vx(P).$ Then $o$ is a convex combination of $-\eps \, x$ and $y$ (with non-zero coefficients). Hence, from strict convexity of $g_l(x)$ and  $g_l(-\eps \, x)>0$, $g_l(x)\ge 0,$ we deduce that $g_l(o) >0,$ a contradiction. Hence $g_l(x)<0$ for every $l \in \natur$ and every $x \in -P_o$ and therefore we may fix $\rho_{o,l}>0$ such that $\bigl( -P_o + \uball^d(o,\rho'_{o,l}) \bigr) \cap (-H_o)$ is disjoint with $\setcond{x \in\aff P}{g_{l_k}(x) \ge 0}.$  Applying the above arguments to all vertices of $P$ we determine quantities $\rho_{v,l}$ and $\rho'_v$ with $v \in \vx(P)$ and $l \in \natur.$ 

The assertion follows by setting 
\begin{eqnarray*}
	l' & := & \max \setcond{l'_v}{v \in \vx (P)}, \\ 
	\rho_v &:=& \min \{ \rho_{v,1},\ldots, \rho_{v,l'}, \rho'_{v}/2 \}
\end{eqnarray*}
for every $v \in\vx(P)$ and taking $\rho:= \min \setcond{\rho_v}{v \in \vx(P)}.$ 

\emph{Part~\ref{exists:AICH}.} We assume $o \notin \aff P.$ Without loss of generality we also assume that $n=d-1$. For $l \in \natur$ the homogeneous continuation  $\Tilde{g}_l(x)$ of $g_l(x)|_{\aff P}$ can be expressed by 
	\begin{multline} \label{08.06.05,10:12}
		\Tilde{g}_l(x) :=  \\ \sprod{x}{u}^{4(l+l_0)^2} - \sum_{v \in \vx(P)} y_{v,l} \left( \frac{1}{|\calF_{d-2}(v,P)|} \sum_{F \in \fcst{d-2}{v}{P} } \Bigl(\bigsprod{x}{u}-\Tilde{q}_F(P,x)\Bigr)^{2 (l+l_0)} \right)^{2(l+l_0)},
	\end{multline}
	where $\Tilde{q}_F(P,x)$ is a homogeneous continuation of $q_F(P,x)|_{\aff P}$, i.e.
	$$
		\Tilde{q}_F(P,x)  := \frac{h(P,u_F(P)) \sprod{u}{x} -\bigsprod{u_F(P)}{x}}{\diam(P)}, 
	$$
	If $\sprod{x}{u} =0,$ then \eqref{08.06.05,10:12} amounts to
	\begin{equation} \label{08.05.05,16:44}
		\Tilde{g}_l(x) :=  - \sum_{v \in \vx(P)} y_{v,k} \left( \frac{1}{|\calF_{d-2}(v,P)|} \sum_{F \in \calF_{d-2}(v,P)} \left(\frac{\bigsprod{u_F(P)}{x}}{\diam(P)} \right)^{2 (l+l_0)} \right)^{2 (l+l_0)}.
	\end{equation}
	Thus, if $\sprod{x}{u}=0$, we have $\Tilde{g}_l(x) \le 0$ with equality if and only if $x=o.$ 
	Directly from the definition of homogeneous continuation it follows
	\begin{equation} \label{08.05.05,16:45}
		\setcond{x \in \real^d}{\sprod{x}{u} > 0, \ \Tilde{g}_l(x) \ge 0} \cup \{o \} = \cone \bigl( \setcond{x \in H}{g_l(x) \ge 0} \bigr).
	\end{equation} 
	Equalities \eqref{08.05.05,16:44} and \eqref{08.05.05,16:45} yield the assertion.
\end{proof} 

We shall need the following observation.
\begin{lemma}  \label{even:degree} 
	The degree of every non-zero concave polynomial over $\real[x]$ is even. \eop
\end{lemma}

\begin{proposition} \label{b:prop} Let $P$ be a $3$-polytope in $\real^3$ and let $v \in \vx(P)$. Then there exists an algorithm which takes $P$ and constructs a sequence of polynomials $\bigl(b_{v,l}(x)\bigr)_{l=1}^{\infty}$ satisfying the following conditions: 
\begin{enumerate}[I.]
	\item \label{even:degree:part} For every $l \in \natur$ the polynomial $b_{v,l}(x)$ is of even degree 
	\item \label{hom:part} For every $l \in \natur$ the polynomial $b_{v,l}(x+v)$ is homogeneous.
	\item \label{08.05.30,10:45} For all $\rho >0$, $l \in \natur$, $I \in \calF_1(v,P)$, and $x \in v+S_\rho(P,v)$ satisfying $\calF^-(P,x) = \calF_2(I,P)$ one has 
	$b_{v,l}(x) \le 0$
	with equality if and only if $x \in \cone(I-v)+v.$ 
	\item \label{conv:cone:part} For every $l \in \natur$ the set 
	\begin{equation} \label{bv:approx} 
		B_l(P,v):= \setcond{ x \in \real^3}{\bigsprod{x-v}{u_v(P)} \le 0, \ b_{v,l}(x) \ge 0}
	\end{equation}
	is a convex cone which has apex at $v$, which satisfies $P \subseteq B_l(P,v).$
	\item \label{approx:cones:part} For every $\rho>0$ 
	$$
		(v+S(P,v)) \cup (v- S(P,v)) \subseteq (b_{v,l})_{\ge 0} \subseteq \{v\} \cup (v+\intr S_{\rho}(P,v)) \cup (v- \intr S_{\rho}(P,v))
	$$
	if $l \in \natur$ is sufficiently large.
\end{enumerate}
\eop
\end{proposition} 
\begin{proof} Replacing $P$ by an appropriate translation we may assume that $v=o.$ Applying Theorem~\ref{tight:interp}(\ref{exists:AICH}), we determine a sequence $\bigl(g_{o,l}(x) \bigr)_{l=1}^{\infty}$ which satisfies $\calA(P_o),$ $\calI(P_o),$ and $\calC(P_o),$ respectively. We define $b_{o,l}(x):=\Tilde{g}_{o,l}(x-v)$, where $\Tilde{g}_{o,l}(x)$ is the homogeneous continuation of $g_{o,l}(x)|_{H_o(P)}$. 

The part of Statement~\ref{even:degree:part} follows Lemma~\ref{even:degree}. Statements~\ref{hom:part} follows by construction. 

Let us show Statement~\ref{08.05.30,10:45}. Consider an arbitrary $\rho > 0$ and $x \in S_\rho(P,v)$ with $\calF^-(P,x)=\calF_2(I,P).$ From the definition of $S_\rho(P,v)$ it follows that $\sprod{x}{u_o(P)} < 0$ and by this $y:= \frac{x}{ |\sprod{x}{u_o(P)}|} \in H_o(P).$ The intersection of $\cone(I)$ and $H_o(P)$ is a vertex of $P_o,$ which we denote by $w.$  Since $\calF^-(P,x) = \calF_2(I,P)$, taking into account Lemma~\ref{08.04.30,11:28}, we obtain $\calF^-(P_o,y)= \calF_2(w,P_o).$ Consequently, $y \in w-S(P_o,w).$ Then, by Theorem~\ref{tight:interp}(\ref{AIC:impl:rho}), we get $g_{o,l}(y) \le 0$ with equality if and only if $y=w.$ Consequently, $\Tilde{g}_{o,l}(x) \le 0$ with equality if and only if $x \in S_\rho(P,o) \cap \aff I = \cone(I).$ 

Statement~\ref{conv:cone:part} follows directly from Theorem~\ref{tight:interp}(\ref{exists:AICH}).

It remains to show Statement~\ref{approx:cones:part}. Consider an arbitrary $\rho>0.$ Since $(g_{o,l}(x))_{l=1}^{\infty}$ satisfies $\calI(P_o)$ and $\calC(P_o)$ and in view of the definition of $b_{o,l}(x)$ we obtain 
\begin{equation*} 
S(P,o) \cup (- S(P,o)) \subseteq (b_{o,l})_{\ge 0}.
\end{equation*} 

By construction, $b_{o,l}(x)=b_{o,l}(-x)$ for every $x \in \real^3.$ Consequently $(b_{o,l})_{\ge 0} = B_l(P,o) \cup (-B_l(P,o)).$ Since $\bigl(g_{o,l}(x) \bigr)_{l=1}^{\infty}$ satisfies $\calA(P_o),$ $\calI(P_o),$ and $\calC(P_o),$ if $l$ is large enough the relation
\begin{equation*} 
	(b_{o,l})_{\ge 0}  \subseteq S_{\rho/2}(P,o) \cup (- S_{\rho/2}(P,o)) 
\end{equation*} 
holds true.  Obviously
\begin{equation*} 
S_{\rho/2}(P,o) \cup (- S_{\rho/2}(P,o)) \subseteq \{o\} \cup \intr S_{\rho}(P,o) \cup (- \intr S_{\rho}(P,o)),
\end{equation*} 
and we arrive at \eqref{bv:approx}.
\end{proof} 

\section{Proof of the main result} \label{sect:proof}

\subsection{The case of bounded polyhedra} \label{sect:proof:bounded}

A statement similar to the following proposition was shown in \Bernig[Section~3.2]; see also the survey \cite{Henk06PolRep}. 

\begin{proposition} \label{main:prop:polyg}
	Let $P$ be a convex polygon in $\real^2$, $p_0(x)$ be a strictly concave polynomial vanishing on each vertex of $P$, and $p_1(x):= \prod_{I \in \calF_1(P)} q_I(P,x).$ Then $P=(p_0,p_1)_{\ge 0}.$ 
\end{proposition} 
\begin{proof} 
	The inclusion $P \subseteq (p_0,p_1)_{\ge 0}$ is trivial. For proving the reverse inclusion we fix an arbitrary $x_0 \in \real^2 \setminus P.$ We distinguish the following three cases. 

	\emph{Case~1: $p_1(x_0)<0$.} Obviously $x_0 \not\in (p_0,p_1)_{\ge 0}.$ 

	\emph{Case~2: $p_1(x_0)=0$.} There exists $I \in \calF_1(P)$ with $q_F(P,x_0)=0.$ The polynomial $p_0(x)$ is strictly concave on $\aff I$, non-negative on $I$, and equal zero at the endpoints of $I.$ The latter easily implies that $p_0(x_0) < 0$. Hence $x_0 \not\in (p_0,p_1)_{\ge 0}.$  
	
	\emph{Case~3:  $p_1(x_0)>0$.} Then $q_{I_1}(P,x_0)<0$ and $q_{I_2}(P,x_0)<0$ for two distinct edges $I_1, \, I_2$ of $P$.  The edges $I_1$ and $I_2$ are not parallel. The intersection point $y$ of $\aff I_1$ and $\aff I_2$ satisfies $q_{I_1}(P,y)=q_{I_2}(P,y)=0$. We fix points $x_1$ and $x_2$ belonging to $I_1$ and $I_2$, respectively, and not coinciding with $y$. By construction, the point $y$ lies in the interior of the triangle $\conv \{x_0, x_1, x_2\}$, that is, $y= \lambda_0 \, x_0 + \lambda_1 \, x_1 + \lambda_2 \, x_2$ for some $\lambda_0, \, \lambda_1, \, \lambda_2  >0$ with $\lambda_0 + \lambda_1 + \lambda_2 =1$.  We show by contradiction, that $p_0(x_0) <0.$ Assume the contrary. Then, by the strict concavity of $p_0(x)$, we obtain 
	$$
		p_0(y) > \sum_{j=0}^2 \lambda_j \, p_0(x_j) \ge 0.
	$$

	The inequality $p_0(y) >0$ contradicts the conclusion made in Case~2 and applied to the point $y$ in place of $x_0$. Hence $p_0(x_0)<0$ and by this $x_0 \not \in (p_0,p_1)_{\ge 0}$. 
\end{proof} 

The following theorem implies Theorem~\ref{main:thm} for the case when $P$ is a $3$-polytope.

\begin{theorem} \label{main:thm:polyt} Let $d \in \{2,3\}$ and $P$ be a $d$-dimensional polytope  in $\real^d$. Let $(g_m(x))_{m=1}^{\infty}$ be a sequence of polynomials satisfying conditions $\calA(P),$ $\calI(P),$ and $\calC(P).$ Then there exists an algorithm that takes $P$ and constructs polynomials $p_1(x),\ldots,p_{d-1}(x) \in \real[x]$ such that $P=(p_0,\ldots,p_{d-1})_{\ge 0}$ and
\begin{equation}  \label{p0:def}
	p_0(x)=g_m(x).
\end{equation}
for some $m \in \natur.$ 
\end{theorem} 
\begin{proof} 
The case $d=2$ follows directly from Theorem~\ref{tight:interp} and Proposition~\ref{main:prop:polyg}. We consider the case $d=3.$ First we show the existence of polynomials $p_0(x), \, p_1(x), \, p_2(x)$ satisfying $P=(p_0,p_1,p_2)_{\ge 0},$ and then we show that these polynomials are constructible. Obviously, there exists $\eps_0>0$ such that 
	\begin{equation} 
		\forall  \, F \in \calF_{d-1}(P) \ \forall \, G \in \calF(P) \  : \  F \cap G = \emptyset  \ \Longrightarrow \ \aff F \cap (G+\uball^d(o,\eps_0)) = \emptyset . \label{eps0:choice} \\
	\end{equation} 

	In view of Theorem~\ref{tight:interp}(\ref{AIC:impl:rho}), condition \eqref{cones:on:vert:cond} is fulfilled for all sufficiently small $\rho>0.$
	We also assume that $\rho$ is sufficiently small so that 
	\eqref{card:2:bound} in Lemma~\ref{rho:lem} is fulfilled.  For $v \in \vx(P)$ and $l \in \natur$ let the polynomials $b_{v,l}(x)$ be as in Proposition~\ref{b:prop}.
Fix $\alpha_l>0$ satisfying 
\begin{equation} \label{alpha:def} 
	|b_{w,l}(x)| \le \alpha_l \qquad \forall \, x \in P + \uball^3(o,\eps_0) \ \forall \, w \in \vx(P).
\end{equation}

We define 
	\begin{eqnarray} 
		p_0(x) & := & g_m(x), \nonumber \\
		p_{1}(x)  & := & \sum_{v \in \vx(P)} f_v(x)^{2k} b_{v,l}(x), \label{p1:def} \\
		p_2(x) & := & \prod_{F \in \calF_2(P)} q_F(P,x), \label{p2:def}
	\end{eqnarray}
	where 
	\begin{equation} \label{fv:def} 
		f_v(x) := \prod_{F \in \calF_2(P) \setminus \calF_2(v,P)}  q_F(P,x) 
	\end{equation}
	for every $v \in \vx(P)$ and the parameters $k, m \in \natur$ will be fixed later. It will be shown that for a sufficiently large $k$ and $m$ we have $P=(p_0,p_1,p_2)_{\ge 0}$. Let us first show that there exists $\eps \in (0,\eps_0]$ such for all sufficiently large $k \in \natur$ we have
	\begin{equation} \label{08.03.06,14:00}
		U(P,\eps,\rho) \cap (p_1,p_2)_{\ge 0} \subseteq P \cup \biggl( \bigcup_{v \in \vx(P)} \bigl( v-\intr S_\rho(P,v) \bigr) \biggr).
	\end{equation} 

	Let us consider an arbitrary $x \in U(P,\eps_0,\rho).$ 

	\emph{Case~1: $x \in U_P(P,\eps_0,\rho)=P$}. Clearly, $x$ belongs to the left and the right hand side of \eqref{08.03.06,14:00}.

	\emph{Case~2: $x \in U_F(P,\eps_0,\rho)$ for some $F \in \calF_2(P)$.} We have $p_2(x) \le 0$ with equality if and only if $x \in F.$ Consequently $U_F(P,\eps_0,\rho) \cap (p_2)_{\ge 0} \subseteq P.$ 
	
	\emph{Case~3: $x \in U_I(P,\eps_0,\rho)$ for some $I \in \calF_1(P)$.}  If $x \in I$ and $v \in \vx(P) \setminus \vx(I),$ then at least one of the two facets $F$ from $\calF_2(I,P)$ satisfies $v \not\in F,$ which yields $f_v(x)=0.$ Consequently
	\begin{equation} \label{08.03.12,11:17}
		\max_{v \in \vx(P) \setminus \vx(I)} |f_v(x)| = 0 \qquad \forall \, x \in I
	\end{equation} 
	 We show that there exists $\beta_I> 0$ is such that 
	\begin{equation} \label{beta:I:bd}
		\beta_I \le \max_{w \in \vx(I)} |f_w(x)| \qquad \forall \, x \in U_I(P,\eps_0,\rho).
	\end{equation}
	Choose $x \in U_I(P,\eps_0,\rho).$ If $x \in \vx(I),$ then $f_w(x) > 0$ for $w=x.$ If $x \in \relint I,$ then $f_w(x)>0$ for every $w \in \vx(I).$ Now assume $x \in U_I(P,\eps_0,\rho) \setminus I.$ We fix arbitrary $w \in \vx(I)$ and $F \in \calF_2(P) \setminus \calF_2(w,P)$.  Consider the subcase $F \cap I= \emptyset.$ By the definition of $U_I(P,\eps_0,\rho)$, we have $x \in I + \uball^3(o,\eps_0),$ and by \eqref{eps0:choice}, we obtain $q_F(P,x) \ne 0.$ Consider the subcase $F \cap I \ne \emptyset.$ We denote by $v$ the endpoint of $I$ distinct from $w.$ Then $F \in \calF_2(v,P)$. Let us show that $q_F(P,x) \ne 0$ by contradiction. Assume the contrary, i.e., $q_F(P,x)=0.$ Then $F \in \calF^-(P,x).$ Furthermore, $\calF_2(I,P) \subseteq \calF^-(P,x)$, by the definition of $U_I(P,\eps_0,\rho).$ Taking into account the above relations together with $\calF_2(I,P)\subseteq \calF_2(v,P)$ and $F \in \calF_2(v,P)$ we deduce
	$\calF_2(I,P) \cup \{F \} \subseteq \calF^-(P,x) \cap \calF_2(v,P).$ Hence $|\calF^-(P,x) \cap \calF_2(v,P)| \ge 3,$ which implies that  \eqref{card:2:bound} cannot be fulfilled, a contradiction.  Summarizing we get 
	$$\max_{w \in \vx(I)} |f_w(x)|>0 \qquad \forall \, x \in U_I(P,\eps_0,\rho),$$
	which yields the existence of $\beta_I$ satisfying \eqref{beta:I:bd}.	By Proposition~\ref{b:prop}(\ref{08.05.30,10:45}) 
	\begin{equation} 
		b_{w,l}(x)  =  0  \qquad \forall \, w \in \vx(I) \ \forall \, x \in I. \label{08.04.22,13:53}
	\end{equation}

	If $x \in \aff I \setminus I,$ one has $x \not\in S_\rho(P,w)$ for some vertex of $w \in \vx(I).$  Consequently $U_I(P,\eps_0,\rho) \cap \aff I = I$.  By Proposition~\ref{b:prop}(\ref{08.05.30,10:45})  we get 
	\begin{equation} 
		b_{w,l}(x)  <  0   \qquad \forall \, w \in \vx(I) \ \forall \, x \in U_I(P,\eps_0,\rho) \setminus I, \label{08.04.22,13:52} 
	\end{equation} 

	We also notice that $\max_{v \in \vx(P) \setminus \vx(I)}{ |f_v(x)|}$ and $b_{w,l}(x)$ are semi-algebraic functions and $U_I(P,\eps_0,\rho)$ is a semi-algebraic set. Hence, taking into account \eqref{08.03.12,11:17}, \eqref{08.04.22,13:53} and \eqref{08.04.22,13:52} and applying Theorem~\ref{Loj},  we obtain  the existence of $k_I \in \natur$ and $\gamma_I>0$ such that
	\begin{equation} \label{08.03.12,11:21}
		\max_{v \in \vx(P) \setminus \vx(I)} f_v(x)^{2k_I}	\le \gamma_I \min_{w \in \vx(I)} |b_{w,l}(x)|
	\end{equation} 
	for every $x \in U_I(P,\eps_0,\rho)$.
	In view of \eqref{08.03.12,11:17}, we can choose $\eps_I \in (0,\eps_0]$ such that for every $x \in U_I(\eps_I,\rho)$ one has 
	\begin{equation} \label{eps:I:def}
		\max_{v \in \vx(P) \setminus \vx(I)} |f_v(x)| \le \frac{\beta_I}{2}.		
	\end{equation}
	
	We also assume that $k$ is large enough so that the inequality 
	\begin{equation} \label{k:I:bound}
		\alpha_l \, |\vx(P)| \, \gamma_I \, 2^{-2k}  \, \left(\frac{2}{\beta_I}\right)^{2k_I}  \le \frac{1}{2} 
	\end{equation}
	is fulfilled.

	Then for every $x \in U_I(P,\rho,\eps_I)$ and every $k \in \natur$ with $k \ge k_I$ we obtain 
	\begin{eqnarray*} 
		& & \biggl| \sum\limits_{v \in \vx(P) \setminus \vx(I)} f_v(x)^{2k} b_{v,l}(x) \biggr|  \\ & \stackrel{\eqref{alpha:def}}{\le} & \alpha_l \sum\limits_{v \in \vx(P) \setminus \vx(I)} f_v(x)^{2k} \\ & \le & \alpha_l \, |\vx(P)| \max\limits_{v \in \vx(P) \setminus \vx(I)} f_v(x)^{2k} 
	 \\ 
	& \stackrel{\eqref{08.03.12,11:21}}{\le} & \alpha_l \, |\vx(P)| \, \gamma_I 	\min\limits_{w \in \vx(I)} |b_{w,l}(x)| \max\limits_{v \in \vx(P) \setminus \vx(I)} f_v(x)^{2(k-k_I)} \\    
	& \stackrel{\eqref{eps:I:def}}{\le} & \alpha_l \, |\vx(P)| \, \gamma_I \,	\left(\frac{\beta_I}{2}\right)^{2(k-k_I)}\,  \min\limits_{w \in \vx(I)} |b_{w,l}(x)|  \\
	& \stackrel{\eqref{beta:I:bd}}{\le} & \alpha_l \, |\vx(P)| \, \gamma_I \, 2^{-2k}  \, \left(\frac{2}{\beta_I}\right)^{2k_I}  \min_{w \in \vx(I)} |b_{w,l}(x)| \max_{w \in \vx(I)} f_w(x)^{2k}  \\
	& \stackrel{\eqref{k:I:bound}}{\le} & \frac{1}{2} \min_{w \in \vx(I)} |b_{w,l}(x)| \max_{w \in \vx(I)} f_w(x)^{2k}  \\
	& \le & \frac{1}{2} \, \min_{w \in \vx(I)} |b_{w,l}(x)| \sum_{w \in \vx(I)} f_w(x)^{2k}   \\ 
	& \le &  \frac{1}{2} \, \sum\limits_{w \in \vx(I)} f_w(x)^{2k} \, |b_{w,l}(x)|.
	\end{eqnarray*} 

	Consequently, for $x$ and $k$ as above, we get
	\begin{eqnarray*}
		p_1(x) & \le & \sum_{w \in \vx(I)} f_w(x)^{2k} b_{w,l}(x) + \biggl| \sum\limits_{v \in \vx(P) \setminus \vx(I)} f_v(x)^{2k} b_{v,l}(x) \biggr| \\
			     & \le & \sum_{w \in \vx(I)} f_w(x)^{2k} b_{w,l}(x) + \frac{1}{2} \, \sum\limits_{w \in \vx(I)} f_w(x)^{2k} \, |b_{w,l}(x)| \\
			    & \stackrel{\eqref{08.04.22,13:52}}{=} & \frac{1}{2} \, \sum\limits_{w \in \vx(I)} f_w(x)^{2k} \, b_{w,l}(x) \stackrel{\eqref{08.04.22,13:52}}{\le} 0
	\end{eqnarray*}
	with equality $p_1(x)=0$ if and only if $x \in I.$ 

	\emph{Case~4: $x \in U'_v(P,\eps_0,\rho)$ for some $ v \in \vx(P)$.} By Proposition~\ref{b:prop}(\ref{approx:cones:part}) and the definition of $U'_v(P,\eps_0,\rho)$ we have
	\begin{equation} \label{08.04.22,15:30}
		b_{v,l}(x) < 0 \qquad \forall \, x \in U'_v(P,\eps_0,\rho) \setminus \{v\}.
	\end{equation} 
	In view of \eqref{eps0:choice}, the definition of $f_v(x)$, and the inclusion $U_v'(P,\eps_0,\rho) \subseteq \uball^3(v,\eps_0)$, there exists $\beta_v>0$ such that 
	\begin{equation} \label{beta:v:def}
		f_v(x) \ge \beta_v \qquad \forall \, x \in U'_v(P,\eps_0,\rho). 
	\end{equation} 
	On the other hand
	\begin{equation} \label{08.03.11,17:37}
		\max_{w \in \vx(P) \setminus \{v\} } |f_w(v)| =0.
	\end{equation} 
	Notice that $\max_{w \in \vx(P) \setminus \{v\} } |f_w(x)|$ and $b_{v,l}(x)$ are semi-algebraic function, and $U'_v(P,\eps_0,\rho)$ is a semi-algebraic set. Thus, taking into account \eqref{08.04.22,15:30} and \eqref{08.03.11,17:37} and applying Theorem~\ref{Loj} we find $\gamma_v>0$ and $k_v \in \natur $ such that
	\begin{equation} \label{08.03.07,17:34}
	\max_{w \in \vx(P) \setminus \{v\} } |f_w(x)|^{2k_v} \le \gamma_v \, |b_{v,l}(x)|
	\end{equation}
	for every $x \in U'_v(P,\eps_0,\rho)$. In view of \eqref{08.03.11,17:37}, we can choose $ \eps_v \in (0,\eps_0]$ such that
	\begin{equation} \label{08.03.12,17:24}
		\max_{w \in \vx(P) \setminus \{v\} } |f_w(x)| \le \frac{\beta_v}{2}
	\end{equation} 
	for every $x \in U'_v(P,\eps_v,\rho)$.
	We assume that $k$ is large enough so that 
	\begin{equation} \label{k:v:bound}
		\alpha_l \,  |\vx(P)| \, \gamma_v \,  \, \left( \frac{\beta_v}{2} \right)^{-2k_v} \, 2^{-2k} \le \frac{1}{2}.
	\end{equation} 
	Then for every $x \in U'_v(P,\eps_v,\rho)$ and $k \ge k_v$ as above we obtain 
	\begin{eqnarray*} 
		\biggl|\sum_{w \in \vx(P) \setminus \{v\}} f_w(x)^{2k} b_{w,l}(x) \biggr| 
		& \stackrel{\eqref{alpha:def}}{\le} & \alpha_l \, |\vx(P)| \, \max\limits_{w \in \vx(P) \setminus \{v\} } |f_w(x)|^{2k}  \\
		& \stackrel{\eqref{08.03.07,17:34}}{\le} & \alpha_l \,  |\vx(P)| \, \gamma_v \, |b_{v,l}(x)| \, \max\limits_{w \in \vx(P) \setminus \{v\} } |f_w(x)|^{2(k-k_v)} \\ 
		& \stackrel{\eqref{08.03.12,17:24}}{\le} & \alpha_l \,  |\vx(P)| \, \gamma_v \, \left( \frac{\beta_v}{2} \right)^{2(k-k_v)} \, |b_{v,l}(x)|  \\ 
		& \stackrel{\eqref{beta:v:def}}{\le} & \alpha_l \,  |\vx(P)| \, \gamma_v  \, \left( \frac{\beta_v}{2} \right)^{2(k-k_v)} \, (\beta_v)^{-2k} f_v(x)^{2k} |b_{v,l}(x)|  \\
		& = & \alpha_l \,  |\vx(P)| \, \gamma_v \,  \, \left( \frac{\beta_v}{2} \right)^{-2k_v} \, 2^{-2k} f_v(x)^{2k} |b_{v,l}(x)| \\
		& \stackrel{\eqref{k:v:bound}}{\le} & \frac{1}{2} f_v(x)^{2k} |b_{v,l}(x)|.
	\end{eqnarray*} 
	Hence, taking into account \eqref{08.04.22,15:30} and \eqref{beta:v:def},  we obtain that  for $k$ satisfying \eqref{k:v:bound} and all $x \in U'_v(P,\eps_v,\rho)$ we have $p_1(x) \le 0$ with equality if and only if $x=v.$ 
	Consequently, for $x$ and $k \ge k_v$ as above we obtain
	\begin{align*}
		p_1(x) & \le  f_v(x)^{2k} b_{v,l}(x) + \biggl|\sum_{w \in \vx(P) \setminus \{v\}} f_w(x)^{2k} b_{w,l}(x) \biggr| \le f_v(x)^{2k}  b_{v,l}(x) + \frac{1}{2} f_v(x)^{2k} |b_{v,l}(x)| \\
			& \stackrel{\eqref{08.04.22,15:30}}{=} \frac{1}{2} f_v(x)^{2k} b_{v,l}(x) \stackrel{\eqref{08.04.22,15:30}}{ \le} 0
	\end{align*}
	with equality $p_1(x)=0$ if and ony if $x=v.$

	\emph{Case~5: $x \in U_v(P,\eps_0,\rho) \setminus U_v'(P,\eps_0,\rho)$ for some $v \in \vx(P).$} By the definition of $U_v(P,\eps_0,\rho)$ and $U_v'(P,\eps_0,\rho)$ we easily see that $x \in v - \intr S_\rho(P,v).$ Thus, $U_v(P,\eps_0,\rho) \setminus U_v'(P,\eps_0,\rho) \subseteq v - \intr S_\rho(P,v)$. 

	By means of the arguments given in the above five cases we verified that \eqref{08.03.06,14:00} holds if $k \ge k_v$, $k \ge k_I$, $\eps \le \eps_v$, $\eps \le \eps_I$ for all $v \in \vx(P)$ and $I \in \calF_1(P)$ and the inequalities \eqref{k:I:bound} and \eqref{k:v:bound} are fulfilled. By Proposition~\ref{U:prop}, there exists $\delta>0$  such that $P+\uball^3(o,\delta) \subseteq U(P,\eps,\rho)$. By condition $\calA(P)$, we can choose $m \in \natur$ such that $(p_0)_{\ge 0} \subseteq P+\uball^3(o,\delta).$ Thus, for $m$ as above we obtain
	\begin{eqnarray*}
		(p_0,p_1,p_2)_{\ge 0} & = & (p_0,p_1,p_2)_{\ge 0} \cap (P+\uball^3(o,\delta)) \subseteq (p_0,p_1,p_2)_{\ge 0} \cap U(P,\eps,\rho)  \\
		& \stackrel{\eqref{08.03.06,14:00}}{\subseteq} & (p_0)_{\ge 0} \cap \biggl( P \cup \bigcup_{v \in \vx(P)} \bigl( v-\intr S_\rho(P,v) \bigr) \biggr) \\
		& = & \left(  (p_0)_{\ge 0} \cap P \right) \cup \biggl( (p_0)_{\ge 0} \cap \bigcup_{v \in \vx(P)} \bigl( v-\intr S_\rho(P,v)  \bigr) \biggr) \\
	   & \stackrel{\eqref{cones:on:vert:cond}}{=} & P \cup \vx(P) = P. 
	\end{eqnarray*}
	Thus $(p_0,p_1,p_2)_{\ge 0} = P$.

	Now let us show the existence of an algorithm constructing $p_1(x), \, p_2(x), \, p_3(x)$ as above. The constructibility of $q_2(x)$ is obvious. Below we show how appropriate polynomials $p_0(x)$ and $p_1(x)$ can be determined. The formulas for $p_0(x)$ and $p_1(x)$ involve the parameters $l, m, k \in \natur$. It is not difficult to construct the sequences $(l_j)_{j=1}^{\infty},$ $(m_j)_{j=1}^{\infty},$ $(k_j)_{j=1}^{\infty}$ such that $$\natur^3 = \setcond{(l_j,m_j,k_j)}{j \in \natur}.$$ We proceed as follows.

	\begin{enumerate}[1:]
		\item Set $j:=1.$
		\item \label{08.06.10,17:25} Set $l:=l_j,$ $m:=m_j,$ $k:=k_j$.
		\item Determine $p_0(x)$ and $p_1(x)$ by \eqref{p0:def} and \eqref{p1:def}, respectively.
		\item \label{08.06.11,09:41} If $P \ne (p_0,p_1,p_2)_{\ge 0}$, set $j:=j+1$ and go to Step~\ref{08.06.10,17:25}.
		\item Return $p_0(x)$ and $p_1(x)$.
	\end{enumerate}

We remark that at Step~\ref{08.06.11,09:41}  the comparison of $P$ and $(p_0,p_1,p_2)_{\ge 0}$ can be performed algorithmically, which follows directly from Theorem~\ref{TarskiSeidenberg}.
\end{proof} 

\subsection{The case of unbounded polyhedra} \label{sect:proof:unbounded}

\begin{proof}[Proof of Theorem~\ref{main:thm}] 
The case when $P$ is bounded follows directly from Theorem~\ref{main:thm:polyt}. Let us consider the case when $P$ is unbounded. Every polyhedron $P$ can be represented as a sum $P=Q+L,$ where $L$ is an affine space and $Q$ is a line-free polyhedron such that $\aff Q$ is orthogonal to $L$; see \eqref{polyt:decomp}. If one can construct polynomials $p_0(x),\ldots,p_{d-1}(x) \in \real[x]$ with $Q= \setcond{x \in \aff Q}{p_0(x) \ge 0,\ldots,p_{d-1}(x) \ge 0},$ then $$P= \setcond{x \in \real^d}{p_0( x \ortproj \aff Q) \ge 0,\ldots, p_{d-1}( x \ortproj \aff Q)},$$ where $x \ortproj \aff Q$ is the orthogonal projection of $x$ onto $\aff Q.$ Thus, also $P$ can be represented by $d$ polynomial inequalities. Consequently, we can restrict ourselves to the case when $P$ is a $d$-dimensional line-free polyhedron in $\real^d$.  

From now on, we replace $P$ by an isometric copy of $P$ in $\real^{d+1}$ and also assume that $o \not\in \aff P$. Then $\homog(P)$ is a $(d+1)$-dimensional pointed polyhedral cone.  Since $\homog(P)$ is pointed, one can determine a hyperplane $H'$ in $\real^{n+1}$ with $o \notin H'$ such that $P':= \homog(P) \cap H'$ is bounded and $\homog(P) = \cone (P').$ We apply Theorem~\ref{tight:interp}(\ref{exists:AICH}) to the polytope $P'$ and construct a sequence of polynomials $(g_l(x))_{l=1}^{\infty}$ satisfying conditions $\calA(P'),$ $\calI(P'),$ $\calC(P'),$ and $\calH(P').$ By Theorem~\ref{main:thm:polyt} one can construct polynomials $f_1(x),\ldots,f_{d-1}(x)$ satisfying 
$$P' = \setcond{x \in H'}{f_0(x)\ge 0, \ldots, f_{d-1}(x) \ge 0}$$
and $f_0(x)=g_l(x)$ for some $l \in \natur$. We choose an affine function $f(x)$ on $H'$ such that $\setcond{x \in H'}{f(x)=0}$ is a $(d-1)$-dimensional affine subspace of $H'$ and such that $f(x)$ is strictly positive on the set $\setcond{x \in H'}{p_0(x) \ge 0}$; the above choice is possible since $\setcond{x \in H'}{p_0(x) \ge 0}$ is bounded. For $j=1,\ldots,d-1$ let us define $k_j:=1$ if $f_j(x)$ has odd degree and $k_j:=0$, otherwise. We set $p_j(x):=f_j(x) \, f(x)^{k_j}$ for $j=1,\ldots,d-1$ and $p_0(x):=f_0(x).$ By construction, the polynomials $p_0(x),\ldots,p_{d-1}(x)$ have even degrees. We have $P' = \setcond{x \in H'}{p_0(x) \ge 0,\ldots, p_{d-1}(x)\ge 0}.$ In fact, if $f(x) >0$, for every $j=0,\ldots,d-1$ the inequality $f_j(x) \ge 0$ is equivalent to the inequality $p_j(x) \ge 0$ and otherwise $p_0(x) <0$, by the choice of $f(x)$ and the definition of $p_0(x)$. For $j=0,\ldots,d-1$ let $\Tilde{p}_j(x)$ be the homogenization of $p_j(x)|_{H'}.$ Let $u' \in \real^d \setminus \{o\}$ denote a normal of $H'.$ By the definition of the homogeneous continuation and the evenness of the degrees of $p_0(x), \ldots, p_{d-1}(x)$, we have 
$$
	\setcond{x \in \real^{d+1}}{\Tilde{p}_0(x) \ge 0,\ldots, \Tilde{p}_{d-1}(x) \ge 0, \ \sprod{x}{u'} \ne 0} \cup \{o \} = \homog(P) \cup ( - \homog(P) ).
$$
Furthermore, since $g_l(x)$ satisfies $\calH(P),$ if $\sprod{x}{u'}=0,$ then $\Tilde{p}_0(x) \le 0$ with equality if and only if on $x=o$. Hence 
$$
	\setcond{x \in \real^{d+1} }{\Tilde{p}_0(x) \ge 0,\ldots, \Tilde{p}_{d-1}(x) \ge 0}  = \homog(P) \cup ( - \homog(P) ),
$$
which implies 
\begin{align*}
	& \setcond{x \in \aff P}{\Tilde{p}_0(x) \ge 0,\ldots, \Tilde{p}_{d-1}(x) \ge 0} = \bigl(\homog(P) \cup ( - \homog(P) )\bigr) \cap (\aff P) \\  = & \bigl( \cone(P) \cup \rec(P) \cup (-\cone(P)) \cup (-\rec(P)) \bigr) \cap (\aff P)= P.
\end{align*}
and we are done.
\end{proof} 

We wish to give another proof of Theorem~\ref{main:thm}. First we formulate a modified version of {\L}ojasiewicz's Inequality. 

\begin{lemma} \label{loj:mod}
	Let $A$ be a bounded and closed semi-algebraic set in $\real^d.$ Let $f(x)$ and $g(x)$ be continuous, semi-algebraic functions on $A$ satisfying $$\setcond{x \in A}{f(x) =0} \subseteq \setcond{x \in A}{g(x) = 0}.$$
	Then there exists $n' \in \natur$ such that for every $n \in \natur$ with $n \ge n'$ there is a constant $\lambda > 0$ satisfying
	\begin{equation*}
		|g(x)|^n \le \lambda \, |f(x)| 
	\end{equation*}
	for every $x \in A.$  
\end{lemma} 
\begin{proof}
	By Theorem~\ref{Loj} we can fix $n' \in \natur$ and $\lambda'>0$ such that $|g(x)|^{n'} \le \lambda' \, |f(x)|$ for every $x \in A.$ Let $\mu \ge 0$ be an upper bound of $|g(x)|$ on $A.$ Then for every $n \in \natur$ with $n \ge n'$ we have $|g(x)|^{n} \le \mu^{n-n'} \, \lambda' \, |f(x)|,$ so that we may set $\lambda:= \mu^{n-n'} \, \lambda'.$ 
\end{proof}

The following proposition (which is also interesting on itself) can be used to give another proof of Theorem~\ref{main:thm}.

\begin{proposition} \label{bd:to:unbd} Let $d \in \natur$ and $d \ge 2.$ Assume that there exists an algorithm that takes an arbitrary $n$-polytope $P$ in $\real^d$ and constructs polynomials $p_0(x),\ldots,p_{n-1}(x) \in \real[x]$ satisfying $P=\setcond{x \in \aff P}{p_0(x) \ge 0,\ldots,p_{n-1}(x) \ge 0}$ and $p_0(x) >0$ for all $x \in P \setminus \vx(P).$ Then there exists an algorithm that takes an arbitrary $d$-polyhedron $P$ in $\real^d$ and constructs $f_0(x),\ldots,f_{d-1}(x) \in \real[x]$ satisfying $P=(f_0,\ldots,f_{d-1})_{\ge 0}.$ 
\end{proposition} 
\begin{proof} Consider an arbitrary unbounded $d$-polyhedron $P$ in $\real^d.$  Using the same arguments as in the beginning of the proof of Theorem~\ref{main:thm} we may restrict ourselves to the case when $P$ is line-free. 

The set  $\homog(P)$ is a $(d+1)$-dimensionial pointed polyhedral cone.  Since $\homog(P)$ is pointed, one can determine a hyperplane $H'$ in $\real^{n+1}$ such that $P':= \homog(P) \cap H'$ is bounded and $\homog(P) = \cone (P').$ Let $u' \in \real^{d+1} \setminus \{o\}$ be the normal of $H'$ with $H' = \setcond{x \in \real^{d+1}}{\sprod{x}{u'} = 1}.$ Let us consider polynomials $p_0(x),\ldots,p_{d-1}(x) \in \real[x]$ such that $$P= \setcond{x \in \aff P'}{p_0(x) \ge 0,\ldots,p_{d-1}(x) \ge 0}$$ 
	and $p_0(x) > 0$ for every $x \in P \setminus \vx(P).$ Without loss of generality we may assume that $p_0(x),\ldots,p_{d-1}(x)$ are homogeneous. We define
	$$
		f(x) := \prod_{v \in \vx(P')} \left( \|x\|^2 \|v\|^2- \sprod{x}{ v }^2 \right).
	$$
	By Lemma~\ref{loj:mod} applied to $f(x)$ and $p_0(x)$ restricted to $P',$ there exist $\lambda >0$ and $l \in \natur$ satisfying $l \, \deg f(x) > \deg p_0(x)$ and $\lambda \, p_0(x) - f(x)^l \ge 0$ for every $x \in P'$ (where $\deg$ stands for degree).

	Define $f_0(x) := \lambda \, p_0(x) \, \sprod{x}{u'}^{k_0} -  f(x)^l$ where $k_0 \in \natur$ is chosen in such a way that $f_0$ is homogeneous, i.e., $k_0$ is determined from the equality 
	\begin{equation} \label{08.06.19,09:47}
		k_0 + \deg p_0(x) = l \, \deg f(x).
	\end{equation}
	We also set $f_i(x) := p_i(x) \sprod{x}{u'}^{k_i}$ for $i=1,\ldots,d-1$, where $k_1,\ldots,k_{d-1} \in \{1,2 \}$ are chosen in such a  way that $f_1(x),\ldots,f_{d-1}(x)$ have even degrees.   By construction, the polynomials $f_0(x),\ldots, f_{d-1}(x)$ are homogeneous, have even degrees, and satisfy
	\begin{equation*}
		\homog(P) \cup (-\homog(P)) = \setcond{x \in \real^{d+1}}{f_0(x) \ge 0,\ldots, f_{d-1}(x) \ge 0} 
	\end{equation*} 
Hence
\begin{align*}
	\setcond{x \in \aff P}{f_0(x) \ge 0,\ldots, \ f_{d-1}(x) \ge 0} = \bigl(\homog(P) \cup ( - \homog(P) )\bigr) \cap (\aff P) = P.
\end{align*}
It remains to show that $\lambda>0$ and $s$ are constructible. One can construct sequences $(\lambda_j)_{j=1}^{\infty}$ and $(l_j)_{j=1}^{\infty}$ such that 
$$
	 \setcond{(\lambda,l)}{\lambda \in \natur, \ l \in \natur, \ l \deg f(x) > \deg p_0(x) }= \setcond{(\lambda_j,l_j)}{j \in \natur}.
$$
Thus, we can proceed as follows
\begin{enumerate}[1:]
	\item Define $k_1,\ldots,k_{d-1}$ as explained above.
	\item Set $j:=1.$
	\item \label{08.06.16,15:22} Set $\lambda:=\lambda_j$ and $l:=l_j.$
	\item Determine $k_0$ from \eqref{08.06.19,09:47}.
	\item Set $f_0(x):=\lambda \, p_0(x) \, \sprod{x}{u'}^{k_0} - f(x)^{l}.$
	\item \label{08.06.16,15:23} If $f_0(x) < 0$ for some $x \in P,$ set $j:=j+1$ and go to Step~\ref{08.06.16,15:22}.
	\item Return $f_0(x)$ and stop.
\end{enumerate} 
In view Theorem~\ref{TarskiSeidenberg}, Step~\ref{08.06.16,15:23} can be implemented algorithmically. 
\end{proof}

Theorem~\ref{main:thm} is a direct consequence of Theorem~\ref{main:thm:polyt} and Proposition~\ref{bd:to:unbd}. The advantage of the proof of Theorem~\ref{main:thm} with the help of Proposition~\ref{bd:to:unbd} is that it does not use much of the structure of the ``input'' polynomials. The disadvantage is that the proof of Proposition~\ref{bd:to:unbd} is less direct, since it uses {\L}ojasiewicz's Inequality and involves an ``exhaustive'' search of appropriate parameters. 

\small

\bibliographystyle{amsalpha}
%\bibliography{\BibPath/poly-rep-polyt}
\def\cprime{$'$} \def\cprime{$'$} \def\cprime{$'$} \def\cprime{$'$}
\providecommand{\bysame}{\leavevmode\hbox to3em{\hrulefill}\thinspace}
\providecommand{\MR}{\relax\ifhmode\unskip\space\fi MR }
\providecommand{\MRhref}[2]{%
  \href{http://www.ams.org/mathscinet-getitem?mr=#1}{#2}
}
\providecommand{\href}[2]{#2}

 \begin{tabular}{l}
        \textsc{Gennadiy Averkov and Martin Henk, 
	} \\ \textsc{Universit\"atsplatz 2,} \textsc{Institut f\"ur Algebra und Geometrie,} \\ \textsc{Fakult\"at f\"ur  Mathematik,} 
        \textsc{Otto-von-Guericke-Universit\"at Magdeburg,} \\ \textsc{D-39106 Magdeburg}	\\
        \emph{e-mails}: \texttt{gennadiy.averkov@googlemail.com}	\\ 
	\phantom{\emph{e-mails:}} \texttt{martin.henk@mathematik.uni-magdeburg.de} \\
	\emph{web}: \texttt{http://fma2.math.uni-magdeburg.de/$\sim$averkov} \\
	\phantom{\emph{web}:} \texttt{http://fma2.math.uni-magdeburg.de/$\sim$henk}
    \end{tabular} 
 \end{document}